\documentclass[a4paper, 11pt, final]{article}

\usepackage{amsfonts,enumerate}
\usepackage{enumerate}
\usepackage{graphicx,psfrag}

\newtheorem{proposition}{Proposition}[section]
\newtheorem{theorem}[proposition]{Theorem}
\newtheorem{lemma}[proposition]{Lemma}
\newtheorem{corollary}[proposition]{Corollary}
\newtheorem{definition}[proposition]{Definition}
\newtheorem{remark}[proposition]{Remark}

\newenvironment{proof}{\smallskip\noindent\emph{Proof.}\hspace{1pt}}%
{\hspace{-5pt}{\nobreak\quad\nobreak\hfill\nobreak$\square$\vspace{8pt}%
\par}\smallskip\goodbreak}

\newenvironment{proofof}[1]{\smallskip\noindent\emph{Proof of #1.}%
\hspace{1pt}}{\hspace{-5pt}{\nobreak\quad\nobreak\hfill\nobreak%
$\square$\vspace{8pt}\par}\smallskip\goodbreak}

\newcommand{\Section}[1]{\section{#1}\setcounter{equation}{0}}

\newcommand{\C}[1]{\mathbf{C^{#1}}}

\newcommand{\Cc}[1]{\mathbf{C_c^{#1}}}

\newcommand{\reali}{{\mathbb{R}}}

\newcommand{\naturali}{{\mathbb{N}}}

\newcommand{\BV}{\mathbf{BV}}

\newcommand{\supp}{\mathop{\mathrm{supp}}}

\renewcommand{\epsilon}{\varepsilon}
\renewcommand{\phi}{\varphi}
\renewcommand{\theta}{\vartheta}
\renewcommand{\O}{{\mathcal O}(1)}
\renewcommand{\L}[1]{\mathbf{L^#1}}

\newcommand{\caratt}[1]{{\chi_{\strut#1}}}

\date{\null}
\begin{document}

\title{A $2$--Phase Traffic Model Based on a Speed Bound}

\author{Rinaldo M.~Colombo$^1$ \and Francesca Marcellini$^2$ \and
  Michel Rascle$^3$} \footnotetext[1]{Universit\`a degli Studi di
  Brescia, Via Branze 38, 25123 Brescia, Italy,
  \texttt{Rinaldo.Colombo@UniBs.it}} \footnotetext[2]{Universit\`{a}
  di Milano--Bicocca, Via Bicocca degli Arcimboldi 8, 20126 Milano,
  Italy, \texttt{F.Marcellini@Campus.Unimib.it}}
\footnotetext[3]{Laboratoire de Mathematiques, U.M.R. C.N.R.S. 6621,
  Universit\`e de Nice, Parc Valrose B.P. 71, F06108 Nice Cedex,
  France, \texttt{Rascle@Math.Unice.fr}}

\maketitle

\begin{abstract}

  \noindent We extend the classical LWR traffic model allowing
  different maximal speeds to different vehicles. Then, we add a
  uniform bound on the traffic speed. The result, presented in this
  paper, is a new macroscopic model displaying $2$ phases, based on a
  non-smooth $2\times 2$ system of conservation laws.  This model is
  compared with other models of the same type in the current
  literature, as well as with a kinetic one. Moreover, we establish a
  rigorous connection between a \textit{microscopic Follow-The-Leader}
  model based on ordinary differential equations and this
  \textit{macroscopic continuum} model.  \medskip

  \noindent\textit{2000~Mathematics Subject Classification:} 35L65,
  90B20

  \medskip


\end{abstract}

\Section{Introduction}
\label{sec:Intro}

Several observations of traffic flow result in underlining two
different behaviors, sometimes called \emph{phases},
see~\cite{Colombo, Edie, Goatin2Phases, Kerner1}. At low density and
high speed, the flow appears to be reasonably described by a function
of the (mean) traffic density. On the contrary, at high density and
low speed, flow is not a single valued function of the density.
\begin{figure}[htpb]
  \centering
  \includegraphics[width=6cm]{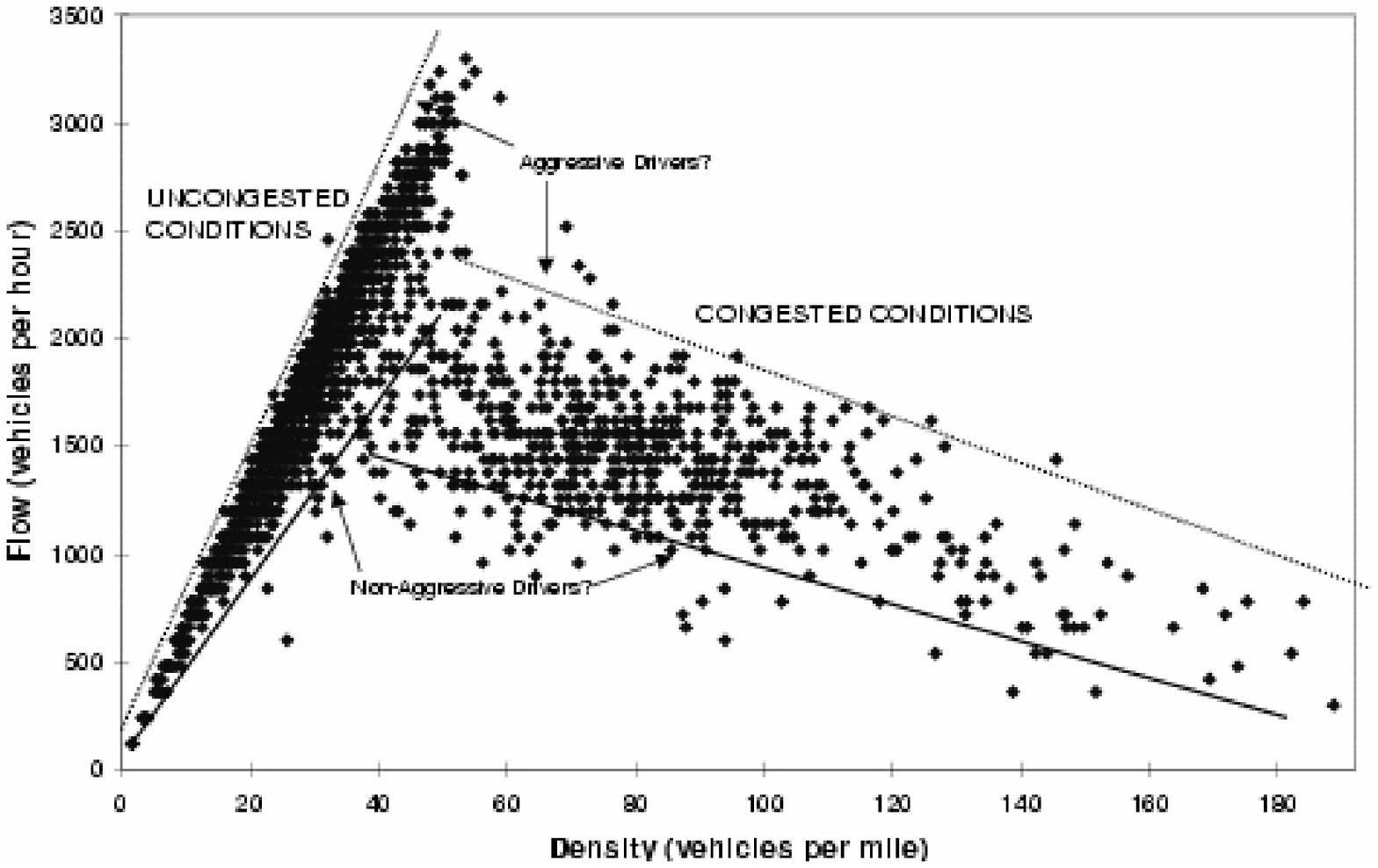}\quad
  \includegraphics[width=6cm]{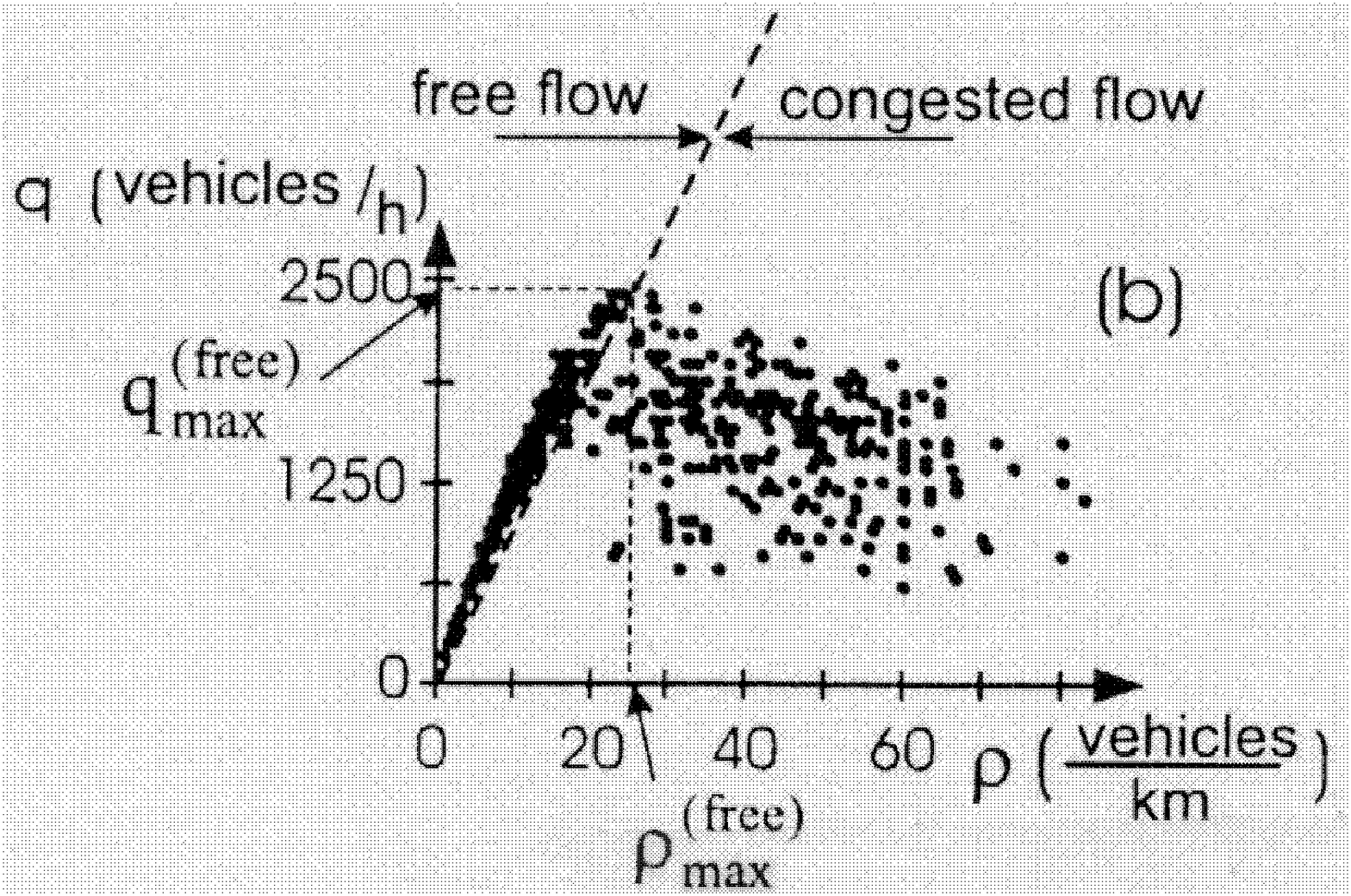}
  \caption{Experimental fundamental diagrams. Left,
    \cite[Figure~1]{Kockelman} and, right, \cite[Figure~1]{Kerner1},
    (see also~\cite{HelbingTreiber}).}
  \label{fig:Ex}
\end{figure}
This paper presents a model providing an explanation to this
phenomenon, its two key features being:
\begin{enumerate}
\item At a given density, different drivers may choose different
  velocities;
\item There exists a uniform bound on the speed.
\end{enumerate}
\noindent By \emph{``bound''}, here we do not necessarily mean an
official speed limit. On the contrary, we assume that different
drivers may have different speeds at the same traffic
density. Nevertheless, there exists a speed $V_{\max}$ that no driver
exceeds. As a result from this postulate, we obtain a fundamental
diagram very similar to those usually observed,
see~Figure~\ref{fig:Ex} and Figure~\ref{fig:phases}, left. Besides,
the evolution prescribed by the model so obtained is reasonable and
coherent with that of other traffic models in the literature. In
particular, we verify that the minimal requirements stated
in~\cite{AwRascle, Daganzo1} are satisfied.

Recall the classical Lighthill-Whitham~\cite{LighthillWhitham} and
Richards~\cite{richards} (LWR) model
\begin{equation}
  \label{eq:LWR}
  \partial_t \rho + \partial_x \left( \rho \, V \right) =0
\end{equation}
for the traffic density $\rho$. Assume that the speed $V$ is not the
same for all drivers. More precisely, different drivers differ in
their \emph{maximal} speed $w$, so that $V = w \, \psi(\rho)$, with
$w\in\left[\check w, \hat w \right],\check w>0 $, being transported
along the road at the mean traffic speed $V$. We identify the
different behaviors of the different drivers by means of their maximal
speed, see also~\cite{BenzoniColombo1,BenzoniColomboGwiazda}. One is
thus lead to study the equations
\begin{equation}
  \label{eq:NonCons}
  \left\{
    \begin{array}{l}
      \partial_t \rho + \partial_x ( \rho v ) =0
      \\
      \partial_t w + v \, \partial_x w =0
    \end{array}
  \right.
  \qquad \mbox{ with } \qquad
  v=w \, \psi(\rho) \,.
\end{equation}
Here, the role of the second equation is to let the maximal velocity
$w$ be propagated with the traffic speed. Indeed, $w$ is a specific
feature of every single driver, 
in other words is a Lagrangian marker.  Therefore this model falls
into the class of models introduced in \cite{AwRascle}, and later on
extended in \cite{LebacqueMammarHajSalem}, see also
\cite[formula~(1.2)]{BagneriniRascle2003}. 

Introducing a uniform bound $V_{\max}$ on the speed, we obtain the
model
\begin{equation}
  \label{eq:rhow}
  \left\{
    \begin{array}{l}
      \partial_t \rho + \partial_x ( \rho v ) =0
      \\
      \partial_t w + v \, \partial_x w =0
    \end{array}
  \right.
  \qquad \mbox{ with } \qquad
  v = \min \left\{V_{\max} ,\, w \, \psi(\rho) \right\}\,.
\end{equation}
We choose to reformulate the above quasilinear system in conservation
form, similarly to~\cite[formula~(1)]{KeyfitzKranzer},
\cite[formula~(2.2)]{BagneriniColomboCorli},
\cite[formula~(1)]{LebacqueMammarHajSalem}, see also~\cite{Temple}, as
follows:
\begin{equation}
  \label{eq:Modeleta}
  \left\{
    \begin{array}{l}
      \partial_t \rho +
      \partial_x \left( \rho\, v (\rho,\eta) \right) = 0
      \\
      \partial_t \eta +
      \partial_x \left( \eta\, v (\rho, \eta) \right) = 0
    \end{array}
  \right.
  \quad \mbox{ with } \quad
  v(\rho, \eta)
  =
  \min \left\{ V_{\max}, \frac{\eta}{\rho}\, \psi(\rho) \right\}
\end{equation}
see the Remark~\ref{rem:appendix} for further comments on this
choice. This model consists of a $2\times 2$ system of conservation
laws with a $\C{0,1}$ but not $\C1$ flow. Note in fact ÃƒÆ’Ã‚Â¹that
$\eta/\rho = w \in [\check w, \hat w]$.  A $2\times 2$ system of
conservation laws with a flow having a similar $\C{0,1}$ regularity is
presented in~\cite{HadelerKuttler} and studied in~\cite{AmadoriShen}.

\par From the traffic point of view, we remark that, under mild
reasonable assumptions on the function $\psi$, the flow
in~(\ref{eq:Modeleta}) may vanish if and only if $\rho = 0$, i.e~the
road is empty, or $\rho = R$, i.e.~the road is fully congested. It is
also worth noting the agreement between experimental fundamental
diagrams often found in the literature and the one related
to~(\ref{eq:Modeleta}), see Figure~\ref{fig:phases}, left.

\par From the analytical point of view, we can extend the present
treatment to the more general case of a maximal speed $V_{\max}$ that
depends on $\rho$, i.e.~$V_{\max} = V_{\max}(\rho)$. However, we
prefer to highlight the main features of the model~(\ref{eq:Modeleta})
in its simplest analytical framework.

As we already said, the model studied here, inspired
from~\cite{Colombo}, falls into the class of ``Aw-Rascle'' models. So
we could use the approach and the theoretical results
of~\cite{AwKlarMaterneRascle}, which should apply here with minor
modifications.

However, our approach is different: here, contrarily to the above
reference, we establish \emph{directly} a connection between the
Follow-The-Leader model in Section~\ref{sec:Micro} and the macroscopic
system~(\ref{eq:Modeleta}), {\em without} viewing both systems as
issued from a same fully discrete system (Godunov scheme) with
different limits, and {\em without} passing in Lagrangian
coordinates. For related works, including vacuum, see also~\cite{Aw,
  GodvikHanche2, GodvikHanche}.  

The present paper is arranged in the following way: in the next
section we study the Riemann Problem for~(\ref{eq:Modeleta}) and
present the qualitative properties of this model from the point of
view of traffic. In Section~\ref{sec:Comp} we compare the present
model with others in the current literature and in
Section~\ref{sec:Micro} we establish the connection with a microscopic
Follow-The-Leader model based on ordinary differential equations. We
also show \textit{rigorously} that the macroscopic
model~(\ref{eq:Modeleta}) can be viewed as the limit of the
microscopic model as the number of vehicles increases to infinity. All
proofs are gathered in the last section.

\Section{Notation and Main Results}
\label{sec:RP}

We assume throughout the following hypotheses:
\begin{enumerate}[{\bf a.}]
\item \label{it:h1} $R,\check w, \hat w, V_{\max}$ are positive
  constants, with $\check w < \hat w$.
\item \label{it:h2} $\psi \in \C{2} \left( [0,R];[0,1]\right)$ is such
  that
  \begin{displaymath}
    \begin{array}{rcl@{\qquad}rcl}
      \psi(0) & = & 1,
      &
      \psi(R) & = & 0,
      \\
      \psi'(\rho) & \leq & 0,
      &
      \displaystyle
      \frac{d^2\ }{d\rho^2} \left( \rho\, \psi(\rho) \right) & \leq &
      0
      \quad \mbox{ for all } \rho \in [0, R]\,.
    \end{array}
  \end{displaymath}
\item \label{it:h3} $\check w > V_{\max}$.
\end{enumerate}
\noindent Here, $R$ is the maximal possible density, typically $R=1$
if $\rho$ is normalized as in Section~\ref{sec:Micro}; $\check w$,
respectively $\hat w$, is the minimum, respectively maximum, of the
maximal speeds of each vehicle; $V_{\max}$ is the overall uniform
upper bound on the traffic speed. At \textbf{\ref{it:h2}.}, the first
three assumptions on $\Psi$ are the classical conditions usually
assumed on speed laws, while the fourth one is technically necessary
in the proof of Theorem~\ref{thm:RP}. The latter condition means that
all drivers do feel the presence of the speed limit.

Moreover, we introduce the notation
\begin{eqnarray}
  \label{eq:phF}
  F
  & = &
  \left\{
    (\rho,w) \in [0,R] \times [\check w, \hat w]
    \colon v(\rho, \rho w) = V_{\max}
  \right\}
  \\
  \label{eq:phC}
  C
  & = &
  \left\{
    (\rho,w) \in [0,R] \times [\check w, \hat w]
    \colon v(\rho, \rho w) = w \, \psi(\rho)
  \right\}
\end{eqnarray}
to denote the \textit{Free} and the \textit{Congested} phases. Note
that $F$ and $C$ are closed sets and $F\cap C \neq \emptyset$.
\begin{figure}[htpb]
  \centering
  \begin{psfrags}
    \psfrag{0}{$0$} \psfrag{rv}{$\rho v$} \psfrag{r}{$\rho$}
    \psfrag{R}{$R$} \psfrag{F}{$F$} \psfrag{C}{$C$}
    \includegraphics[width=3.8cm]{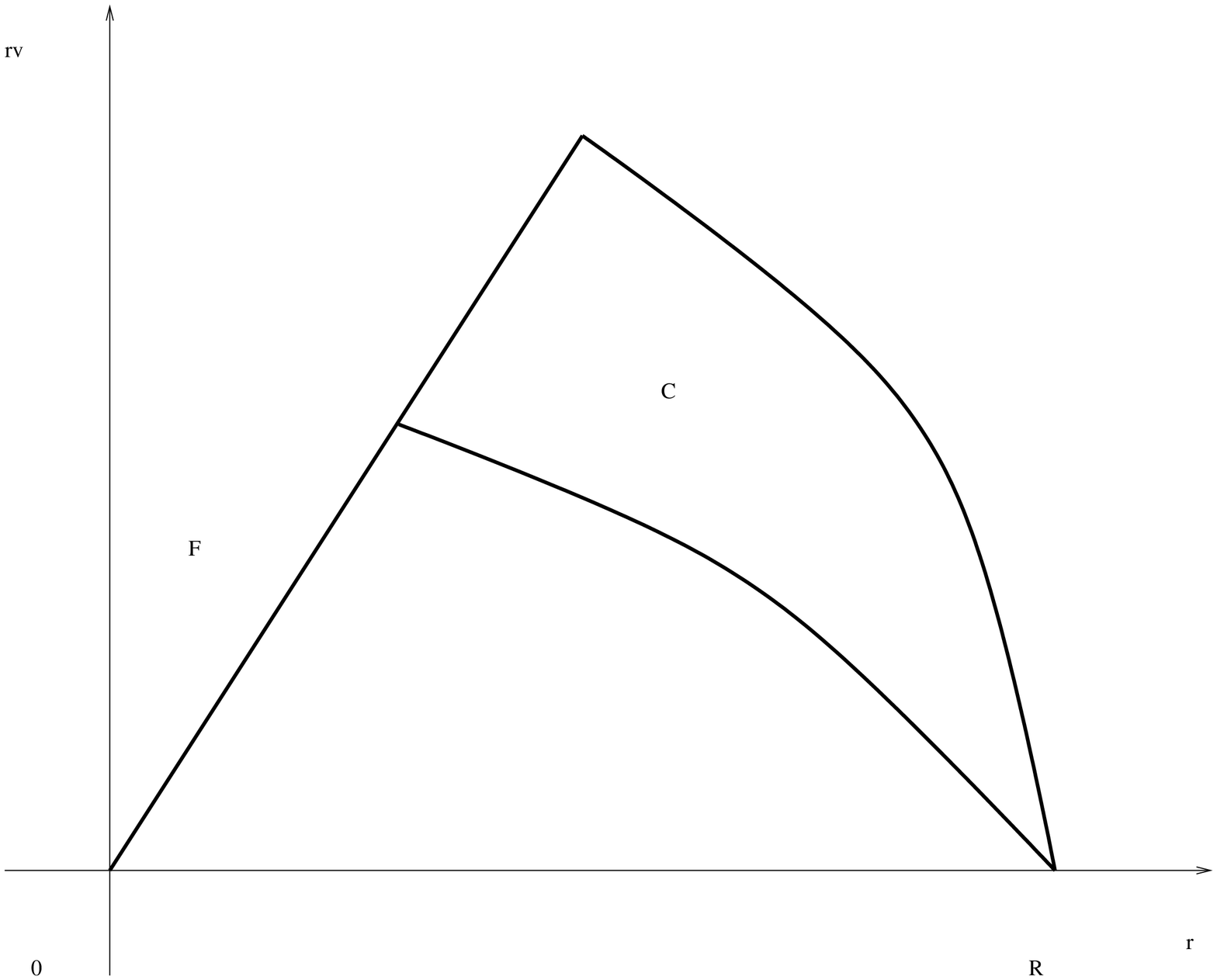}
  \end{psfrags}
  \hfil
  \begin{psfrags}
    \psfrag{0}{$0$} \psfrag{rv}{$w$} \psfrag{r}{$\rho$}
    \psfrag{R}{$R$} \psfrag{F}{$F$} \psfrag{C}{$C$}
    \psfrag{V}{$\!\!\!\!\!\!\!\!\!V_{\max}$} \psfrag{ws}{$\hat w$}
    \psfrag{wg}{$\check w$}
    \includegraphics[width=3.8cm]{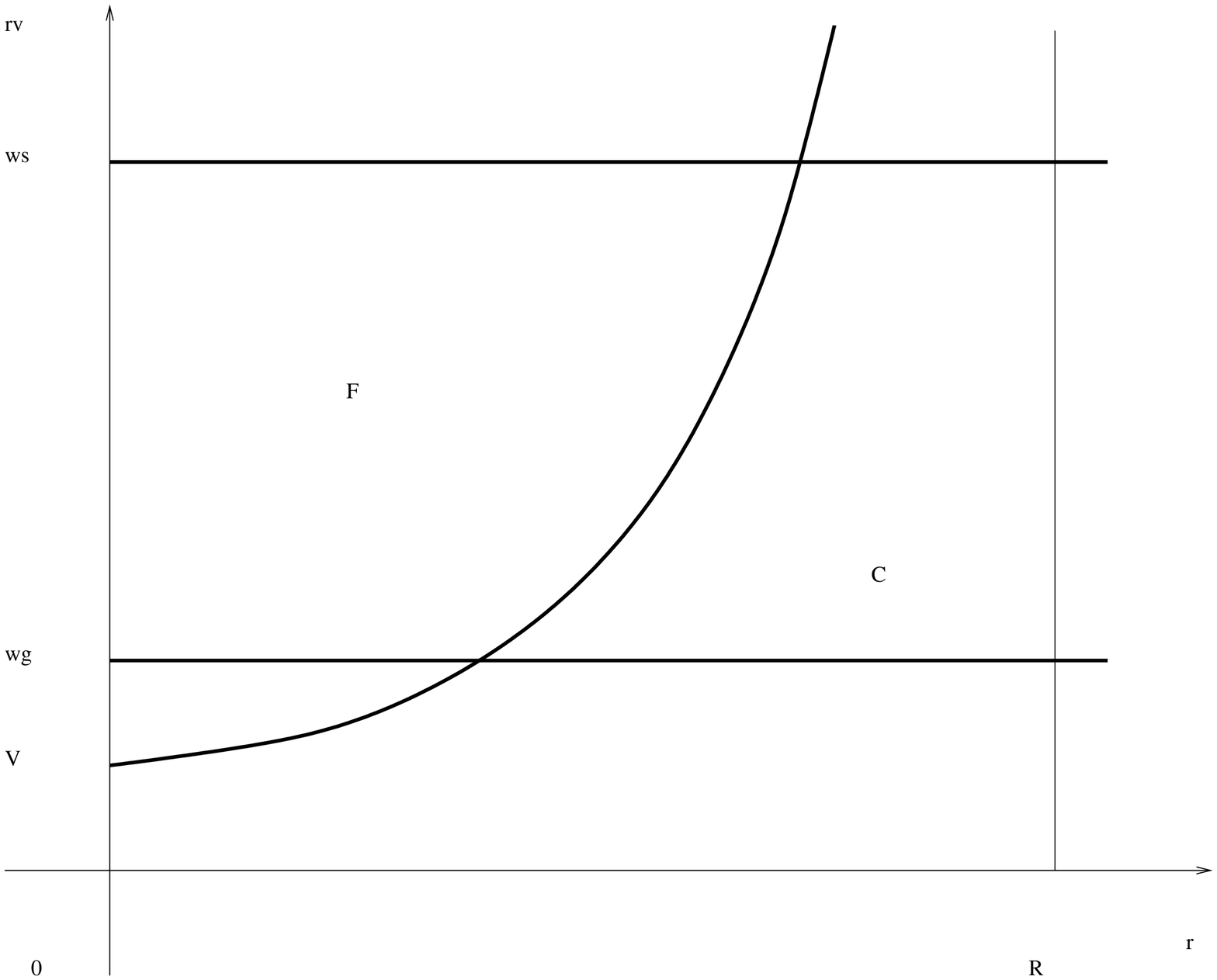}
  \end{psfrags}
  \hfil
  \begin{psfrags}
    \psfrag{0}{$0$} \psfrag{rv}{$\eta$} \psfrag{r}{$\rho$}
    \psfrag{R}{$R$} \psfrag{F}{$F$} \psfrag{C}{$C$}
    \includegraphics[width=3.8cm]{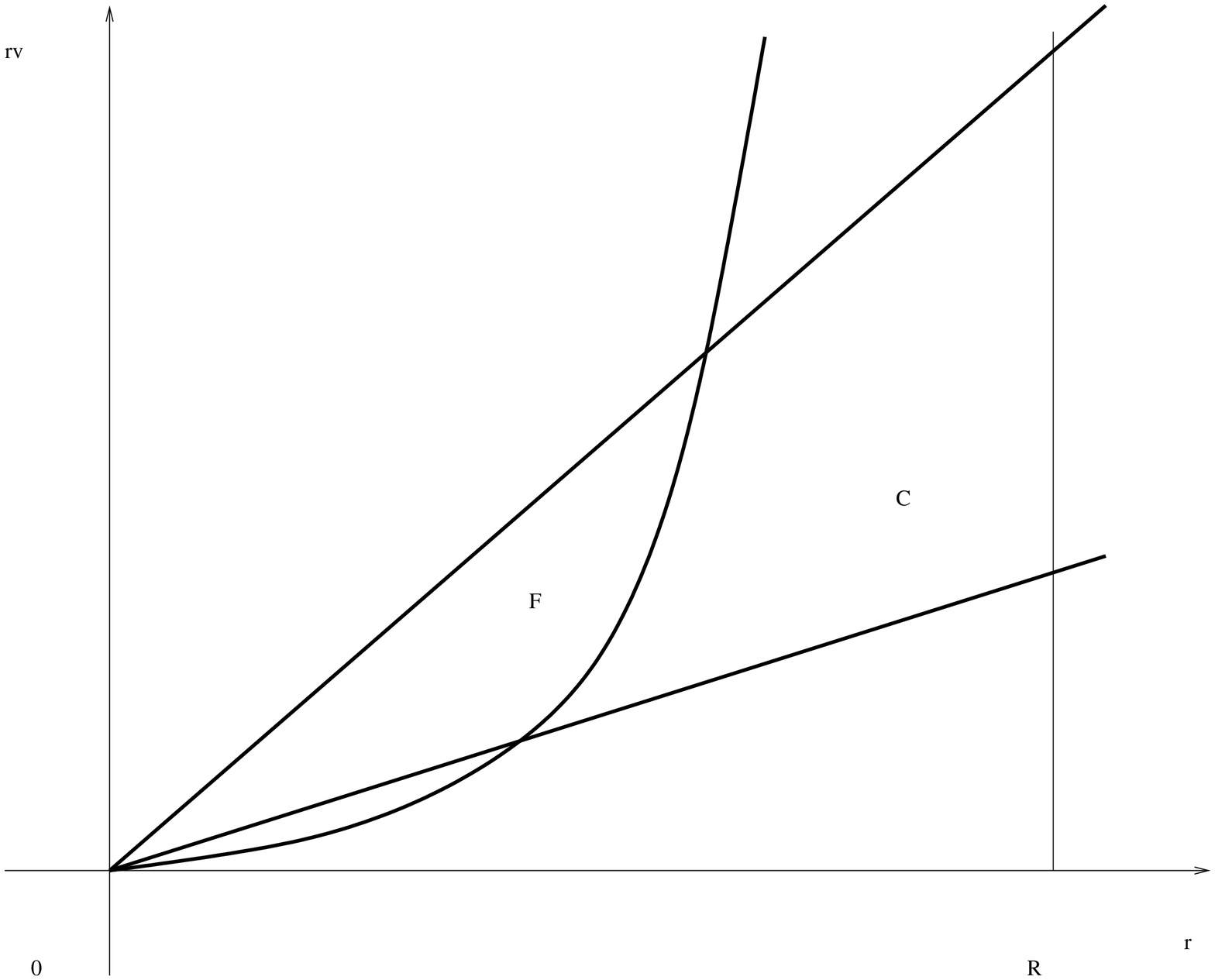}
  \end{psfrags}
  \caption{The phases $F$ and $C$ in the coordinates, from left to
    right, $(\rho, \rho v)$, $(\rho, w)$ and $(\rho, \eta)$.}
  \label{fig:phases}
\end{figure}
Note also that $F$ is $1$--dimensional in the $(\rho, \rho v)$ plane
of the fundamental diagram, while it is $2$--dimensional in the
$(\rho,w)$ and $(\rho,\eta)$ coordinates,
see~Figure~\ref{fig:phases}. See also~Figure~\ref{fig:phase} to have a
vision in three dimensions.
\begin{figure}[htpb]
  \centering
  \begin{psfrags}
    \psfrag{r}{$\rho$} \psfrag{g}{$\rho v$} \psfrag{w}{$w$}
    \psfrag{k}{$u^{l}$} \psfrag{j}{$u^{m}$} \psfrag{l}{$u^{r}$}
    \psfrag{f}{$F$} \psfrag{c}{$C$} \psfrag{b}{$\hat
      w$}\psfrag{e}{$\check w$} \psfrag{z}{$R$}
    \includegraphics[width=8cm]{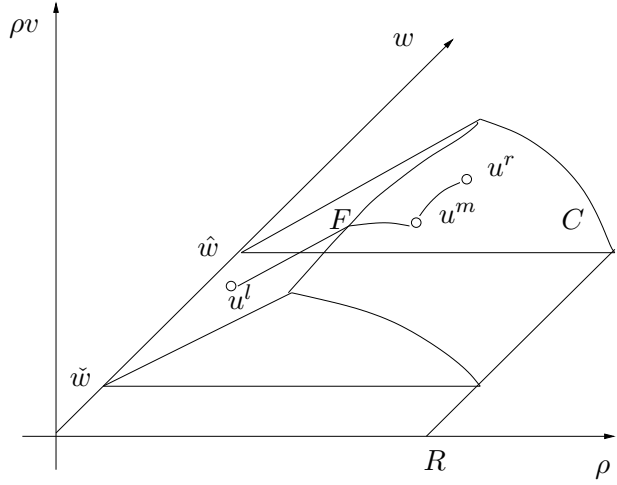}
  \end{psfrags}
  \caption{The phases $F$ and $C$ in the coordinates $(\rho, \rho v,
    w)$. Note that $F$ is contained in a plane. This figure shows an
    example of Riemann Problem when $u^{l}=(\rho^{l}, \rho^{l} v^{l},
    w^{l})\in F$ and $u^{r}=(\rho^{r}, \rho^{r} v^{r}, w^{r})\in C$.}
  \label{fig:phase}
\end{figure}

Let $\rho_*$ be the maximum of the points of maximum of the flow,
i.e.~$\rho_* = \max \left\{ \rho \in [0,R] \colon \rho \, \psi(\rho) =
  \max_{r \in [0,R]} r \, \psi(r) \right\}$. Then, the condition
\begin{equation}
  \label{eq:CapacityDrop}
  \hat w \psi\left(\rho_{*}\right) \geq V_{\max}
\end{equation}
is a further reasonable assumption. Indeed, it means that the maximum
flow is attained in the free phase, coherently with the \emph{capacity
  drop} phenomenon, see for
instance~\cite{TreiberKestingHelbing}. However,
(\ref{eq:CapacityDrop}) is not necessary in the following results.

\medskip

Our next goal is to study the Riemann Problem for~(\ref{eq:Modeleta}).

\begin{theorem}
  \label{thm:RP}
  Under the assumptions~\textbf{\ref{it:h1}.}, \textbf{\ref{it:h2}.}
  and~\textbf{\ref{it:h3}.}, for all states $(\rho^l,\eta^l)$,
  $(\rho^r, \eta^r) \in F \cup C$, the Riemann problem consisting
  of~(\ref{eq:Modeleta}) with initial data
  \begin{equation}
    \label{eq:RD}
    \rho(0,x) = \left\{
      \begin{array}{l@{\quad\mbox{ if }}rcl}
        \rho^l & x & < & 0
        \\
        \rho^r & x & > & 0
      \end{array}
    \right.
    \qquad
    \eta(0,x) = \left\{
      \begin{array}{l@{\quad\mbox{ if }}rcl}
        \eta^l & x & < & 0
        \\
        \eta^r & x & > & 0
      \end{array}
    \right.
  \end{equation}
  admits a unique self similar weak solution $(\rho,\eta) =
  (\rho,\eta) (t,x)$ constructed as follows:
  \begin{enumerate}[(1)]
  \item If $(\rho^l,\eta^l), (\rho^r,\eta^r) \in F$, then
    \begin{equation}
      \label{eq:solF}
      (\rho, \eta) (t,x) =
      \left\{
        \begin{array}{l@{\quad\mbox{ if }\quad}rcl}
          (\rho^l, \eta^l) & x & < & V_{\max}t
          \\
          (\rho^r, \eta^r) & x & > & V_{\max}t \,.
        \end{array}
      \right.
    \end{equation}
  \item If $(\rho^l,\eta^l), (\rho^r,\eta^r) \in C$, then
    $(\rho,\eta)$ consists of a $1$--Lax wave (shock or rarefaction)
    between $(\rho^l, \eta^l)$ and $(\rho^m, \eta^m)$, followed by a
    $2$--contact discontinuity between $(\rho^m, \eta^m)$ and
    $(\rho^r, \eta^r)$. The middle state $(\rho^m, \eta^m)$ is in $C$
    and is uniquely characterized by the two conditions $\eta^m/\rho^m
    = \eta^l / \rho^l$ and $v(\rho^m, \eta^m) = v(\rho^r, \eta^r)$.
  \item If $(\rho^l,\eta^l) \in C$ and $(\rho^r,\eta^r) \in F$, then
    the solution $(\rho,\eta)$ consists of a rarefaction wave
    separating $(\rho^r, \eta^r)$ from a state $(\rho^m, \eta^m)$ and
    by a linear wave separating $(\rho^m, \eta^m)$ from $(\rho^l,
    \eta^l)$. The middle state $(\rho^m, \eta^m)$ is in $F \cap C$ and
    is uniquely characterized by the two conditions $\eta^m/\rho^m =
    \eta^r / \rho^r$ and $v(\rho^m, \eta^m) = V$.
  \item If $(\rho^l,\eta^l) \in F$ and $(\rho^r,\eta^r) \in C$, then
    $(\rho, \eta)$ consists of a shock between $(\rho^l, \eta^l)$ and
    $(\rho^m, \eta^m)$, followed by a contact discontinuity between
    $(\rho^m, \eta^m)$ and $(\rho^r, \eta^r)$. The middle state
    $(\rho^m, \eta^m)$ is in $C$ and is uniquely characterized by the
    two conditions $\eta^m/\rho^m = \eta^l / \rho^l$ and $v(\rho^m,
    \eta^m) = v(\rho^r, \eta^r)$.
  \end{enumerate}
\end{theorem}

\noindent (If $\frac{d^2}{d\rho^2}\left(\rho\,\psi(\rho)\right)$
vanishes, then the words \emph{``shock''} and \emph{``rarefaction''}
above have to be understood as \emph{``contact discontinuities'').}

We now pass from the solution to single Riemann problems to the
properties of the \emph{Riemann Solver}, i.e.~of the map $\mathcal{R}
\colon (F \cup C)^2 \to \BV(\reali; C \cup F)$ such that $\mathcal{R}
\left( (\rho^l, \eta^l), (\rho^r, \eta^r) \right)$ is the solution
to~(\ref{eq:Modeleta})--(\ref{eq:RD}) computed at time, say, $t=1$.

To this aim, recall the following definition, see~\cite{Colombo}:

\begin{definition}
  \label{def:consistent}
  A Riemann Solver $\mathcal{R}$ is \emph{consistent} if the following
  two conditions hold for all $(\rho^l,\eta^l)$, $(\rho^m,\eta^m)$,
  $(\rho^r,\eta^r) \in F \cup C$, and $\bar x \in \reali$:
  \begin{description}
  \item[(C1)] If $\mathcal{R} \left( (\rho^l,\eta^l),(\rho^m,\eta^m)
    \right)(\bar x) = (\rho^m,\eta^m)$ and $\mathcal{R} \left(
      (\rho^m,\eta^m),(\rho^r,\eta^r) \right)(\bar x)$ $ =
    (\rho^m,\eta^m)$, then
    \begin{displaymath}
      \mathcal{R} \left( (\rho^l,\eta^l),(\rho^r,\eta^r) \right) ~ = \left\{\!\!
        \begin{array}{l@{}l}
          \mathcal{R} \left( (\rho^l,\eta^l),(\rho^m,\eta^m) \right)\,,
          & \hbox{if } x < \bar x\,,
          \\
          \mathcal{R} \left( (\rho^m,\eta^m),(\rho^r,\eta^r) \right)\,,
          & \hbox{if } x \geq \bar x\,,
        \end{array}
      \right.
    \end{displaymath}
  \item[(C2)] If $\mathcal{R} \left( (\rho^l,\eta^l),(\rho^r,\eta^r)
    \right) (\bar x) = {(\rho^m,\eta^m)}$, then
    \begin{displaymath}
      \begin{array}{c}
        \mathcal{R} \left( (\rho^l,\eta^l),(\rho^m,\eta^m) \right) = \left\{\!\!
          \begin{array}{l@{}l}
            \mathcal{R} \left( (\rho^l,\eta^l),(\rho^r,\eta^r) \right)\,,
            & \hbox{ if } x\leq \bar x\,,
            \\
            (\rho^m,\eta^m) \,,& \hbox{ if } x > \bar x\,,
          \end{array}
        \right.
        \\
        \mathcal{R} \left( (\rho^m,\eta^m),(\rho^r,\eta^r) \right)
        =\! \left\{\!\!
          \begin{array}{l@{}l}
            (\rho^m,\eta^m)\,, & \hbox{ if } x < \bar x\,,
            \\
            \mathcal{R} \left( (\rho^l,\eta^l),(\rho^r,\eta^r) \right)\,,
            &
            \hbox{ if } x \geq \bar x\,.
          \end{array}
        \right.
      \end{array}
    \end{displaymath}
  \end{description}
\end{definition}
Essentially, \textbf{(C1)} states that whenever two solutions to two
Riemann problems can be placed side by side, then their juxtaposition
is again a solution to a Riemann problem.
\begin{figure}[htpb]
  \centering
  \begin{psfrags}
    \psfrag{t}{$t$} \psfrag{x}{$x$}
    \psfrag{ul}{$\scriptstyle(\rho^l,\eta^l)$}
    \psfrag{um}{$\!\!\!\scriptstyle(\rho^m,\eta^m)$}
    \includegraphics[width=4cm]{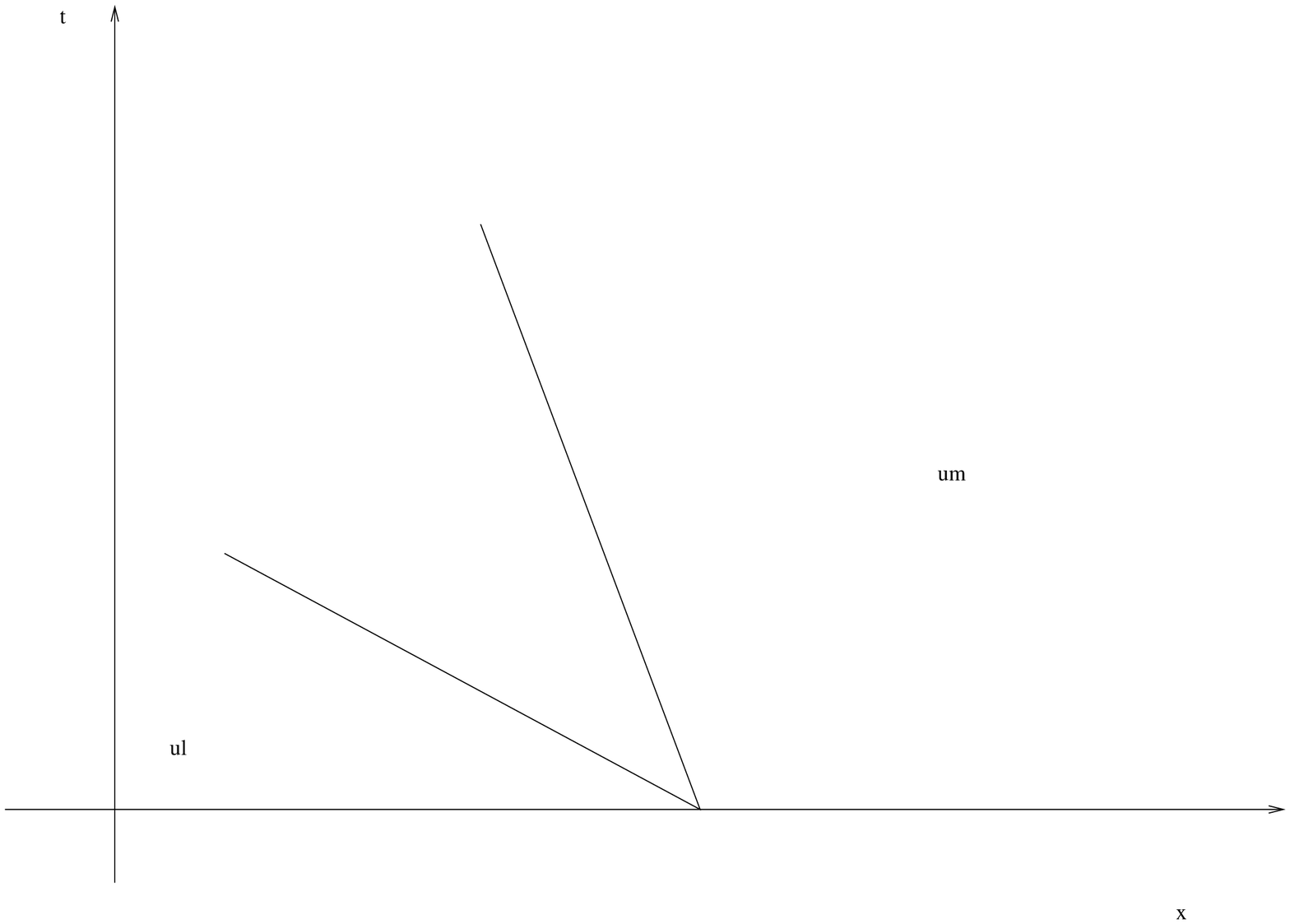}
  \end{psfrags}
  \begin{psfrags}
    \psfrag{t}{$t$} \psfrag{x}{$x$}
    \psfrag{ur}{$\!\!\!\scriptstyle(\rho^r,\eta^r)$}
    \psfrag{um}{$\!\!\!\scriptstyle(\rho^m,\eta^m)$}
    \includegraphics[width=4cm]{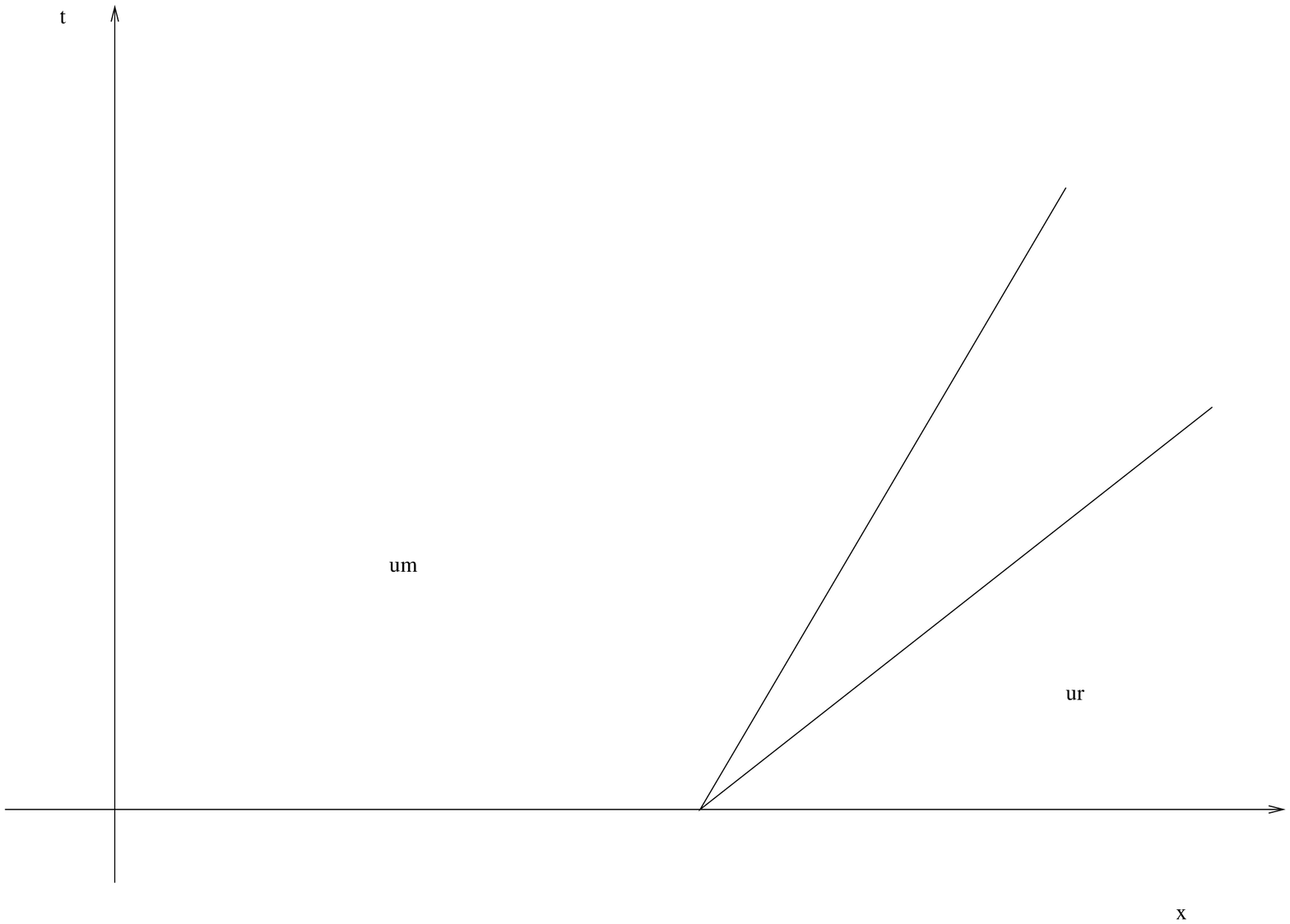}
  \end{psfrags}
  \begin{psfrags}
    \psfrag{t}{$t$} \psfrag{x}{$x$}
    \psfrag{ur}{$\scriptstyle(\rho^r,\eta^r)$}
    \psfrag{um}{$\scriptstyle(\rho^m,\eta^m)$}
    \psfrag{ul}{$\scriptstyle(\rho^l,\eta^l)$}
    \includegraphics[width=4cm]{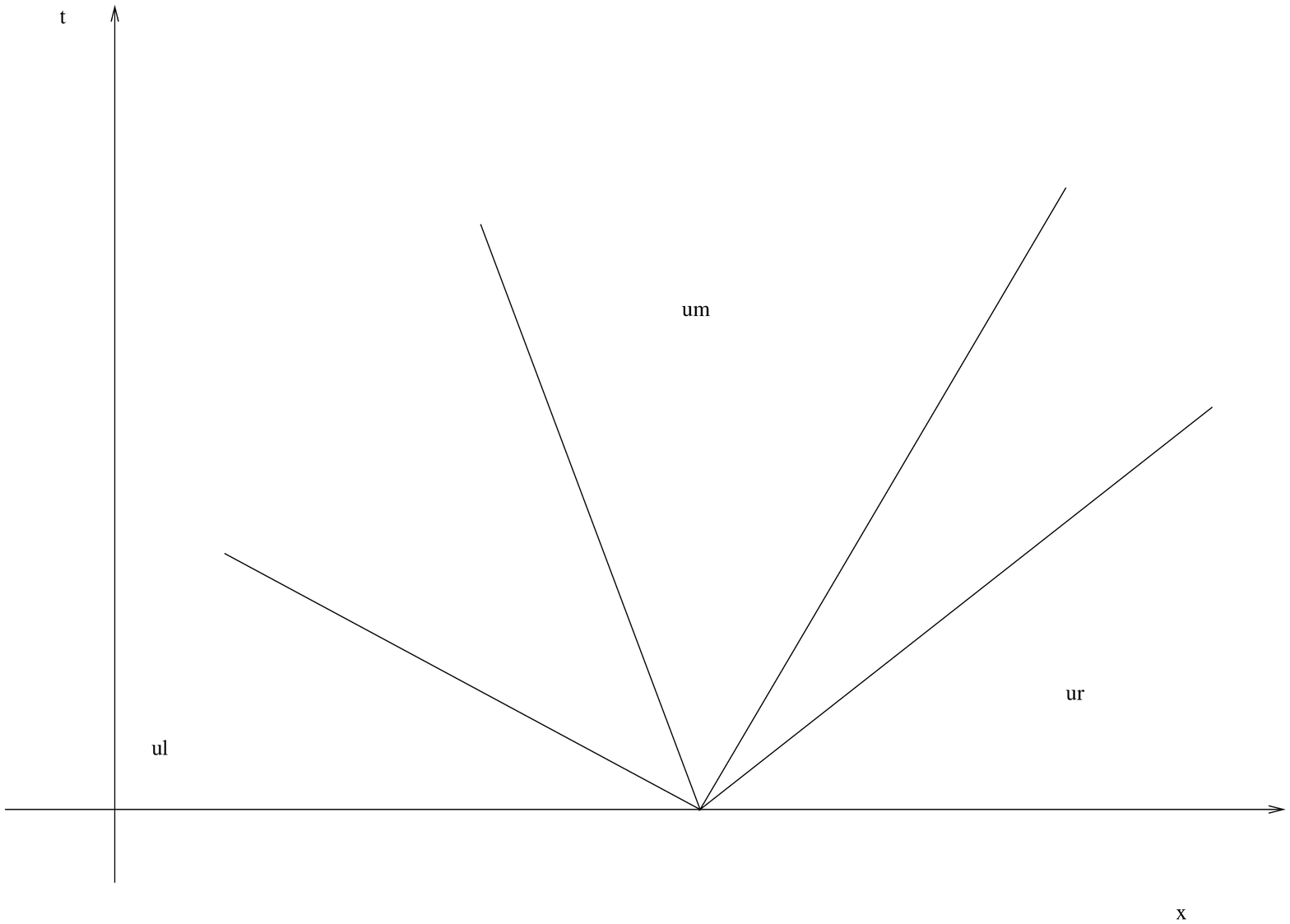}
  \end{psfrags}
  \caption{The conditions~\textbf{(C1)} and~\textbf{(C2)}.}
\end{figure}
Condition~\textbf{(C2)} is the vice-versa.

\medskip

The next result characterizes the Riemann Solver defined above.

\begin{proposition}
  \label{prop:RS}
  Let the assumptions~\textbf{\ref{it:h1}.}, \textbf{\ref{it:h2}.}
  and~\textbf{\ref{it:h3}.} hold. The Riemann Solver $\mathcal{R}$
  defined in Theorem~\ref{thm:RP} enjoys the following three
  conditions
  \begin{enumerate}[1.]
  \item It is consistent in the sense of
    Definition~\ref{def:consistent}.
  \item If $(\rho^l,\eta^l), (\rho^r,\eta^r) \in F$, then $\mathcal{R}
    \left((\rho^l,\eta^l), (\rho^r,\eta^r) \right)$ is the standard
    solution to the linear system
    \begin{equation}
      \label{eq:F}
      \left\{
        \begin{array}{l}
          \partial_t \rho +
          \partial_x \left( \rho\, V_{\max} \right) = 0
          \\
          \partial_t \eta +
          \partial_x \left( \eta V_{\max}\right) = 0,
        \end{array}
      \right.
    \end{equation}
  \item If $(\rho^l,\eta^l) \in F \cup C$ and $(\rho^r,\eta^r) \in C$,
    then $\mathcal{R} \left((\rho^l,\eta^l), (\rho^r,\eta^r) \right)$
    is the standard Lax solution to
    \begin{equation}
      \label{eq:C}
      \left\{
        \begin{array}{l}
          \partial_t \rho +
          \partial_x \left( \eta \, \psi(\rho) \right) = 0
          \\
          \partial_t \eta +
          \partial_x \left(
            \frac{\eta^{2}}{\rho} \, \psi (\rho)
          \right)
          =
          0 \,.
        \end{array}
      \right.
    \end{equation}
  \end{enumerate}
  Moreover, the conditions~\textbf{(C1)}, 2. and 3. uniquely
  characterize the Riemann Solver $\mathcal{R}$.
\end{proposition}

\noindent The above properties are of use, for instance, in using
model~(\ref{eq:Modeleta}) on traffic networks, according to the
techniques described in~\cite{GaravelloPiccoli}.

The next result presents the relevant qualitative properties of the
Riemann Solver defined in Theorem~\ref{thm:RP} from the point of view
of traffic.

\begin{proposition}
  \label{prop:Q}
  Let the assumptions~\textbf{\ref{it:h1}.}, \textbf{\ref{it:h2}.}
  and~\textbf{\ref{it:h3}.} hold. Then, the Riemann Solver
  $\mathcal{R}$ enjoys the following properties:
  \begin{enumerate}
  \item If the initial datum attains values in $F$, $C$, or $F \cup C$
    then, respectively, the solution attains values in $F$, $C$, or $F
    \cup C$.
  \item Traffic density and speed are uniformly bounded.
  \item Traffic speed vanishes if and only if traffic density is
    maximal.
  \item No wave in the solution to~(\ref{eq:Modeleta})--(\ref{eq:RD})
    may travel faster than traffic speed, i.e.~information is carried
    by vehicles.
  \end{enumerate}
\end{proposition}

\Section{Comparison with Other Macroscopic Models}
\label{sec:Comp}

This section is devoted to compare the present
model~(\ref{eq:Modeleta}) with a sample of models from the
literature. In particular, we consider differences in the number of
free parameters and functions, in the fundamental diagram and in the
qualitative structures of the solutions.  Recall that the evolution
described by model~(\ref{eq:Modeleta}) and the corresponding invariant
domain depends on the function $\psi$ and on the parameters
$V_{\max}$, $R$, $\check w$ and $\hat w$. The fundamental diagram
of~(\ref{eq:Modeleta}) is in Figure~\ref{fig:phases}, left.

\subsection{The LWR Model}

In the LWR model~(\ref{eq:LWR}), a suitable speed law has to be
selected, analogous to the choice of $\psi$
in~(\ref{eq:Modeleta}). Besides, in~(\ref{eq:Modeleta}) we also have
to set $V_{\max}$, $R$ and the two geometric positive parameters
$\check w$ and $\hat w$.

The fundamental diagram of~(\ref{eq:Modeleta}) seems to better agree
with experimental data than that of~(\ref{eq:LWR}), shown in
Figure~\ref{fig:fd}, left.
\begin{figure}[htpb]
  \centering
  \begin{psfrags}
    \psfrag{rhov}{$\rho\, v$} \psfrag{0}{$0$} \psfrag{rmax}{$R$}
    \psfrag{rho}{$\rho$}
    \includegraphics[width=4cm]{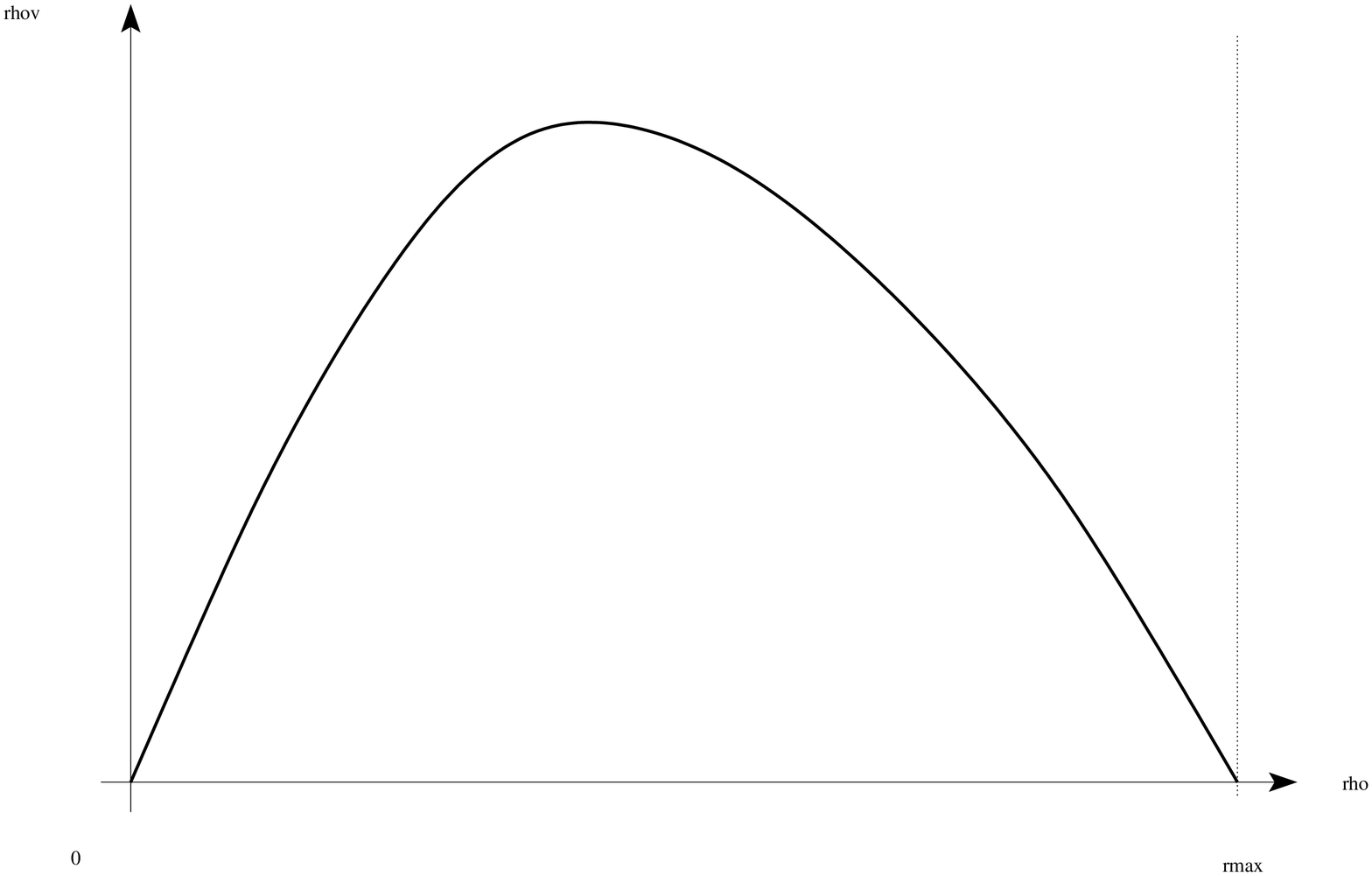}
  \end{psfrags} \begin{psfrags} \psfrag{rhov}{$\rho\, v$}
    \psfrag{0}{$0$} \psfrag{rmax}{$R$} \psfrag{rho}{$\rho$}
    \includegraphics[width=4cm]{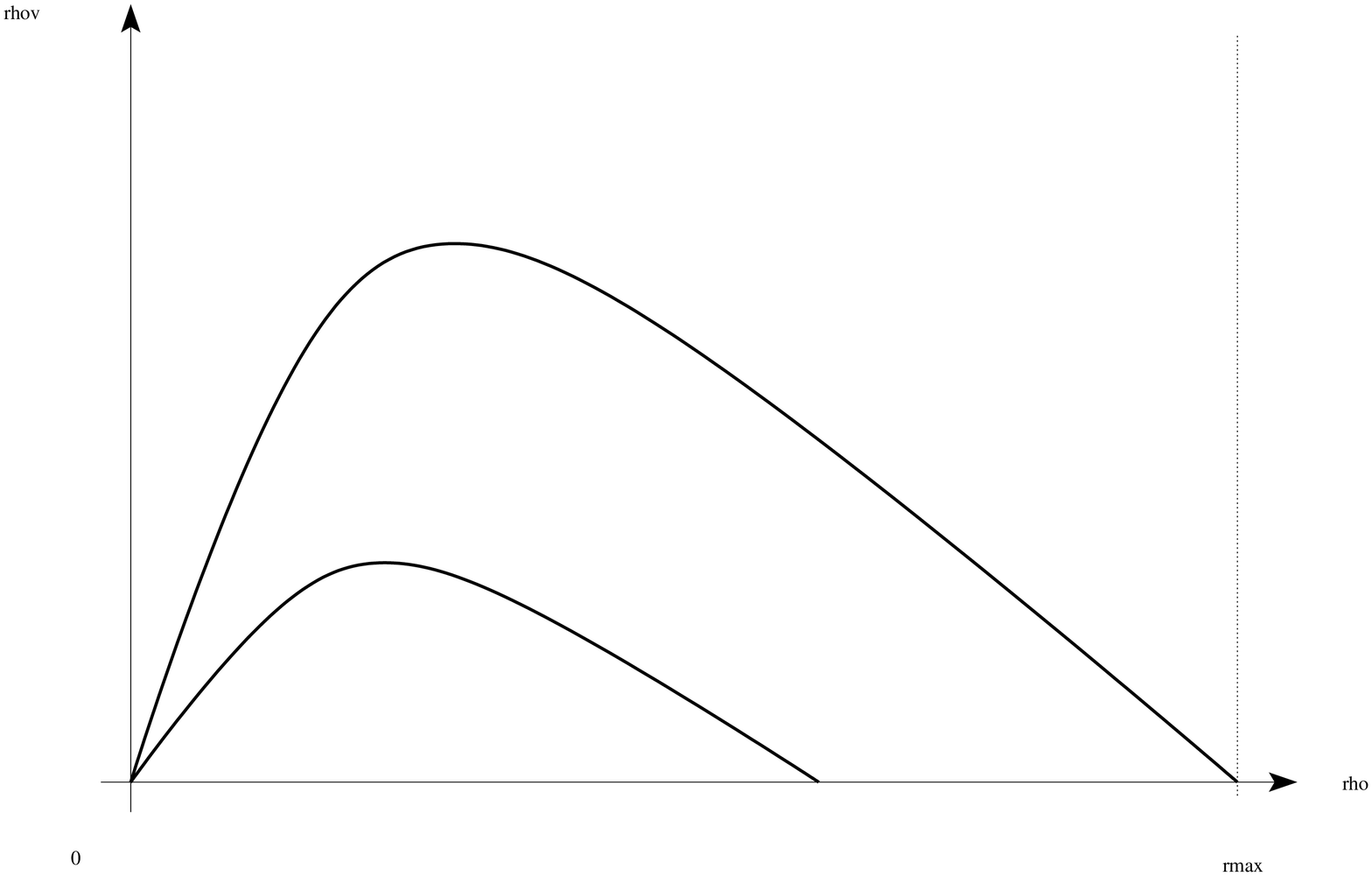}
  \end{psfrags} \begin{psfrags} \psfrag{rhov}{$\rho\, v$}
    \psfrag{0}{$0$} \psfrag{rmax}{$R$} \psfrag{rho}{$\rho$}
    \psfrag{Free}{\quad$F$} \psfrag{congested}{$C$}
    \includegraphics[width=4cm]{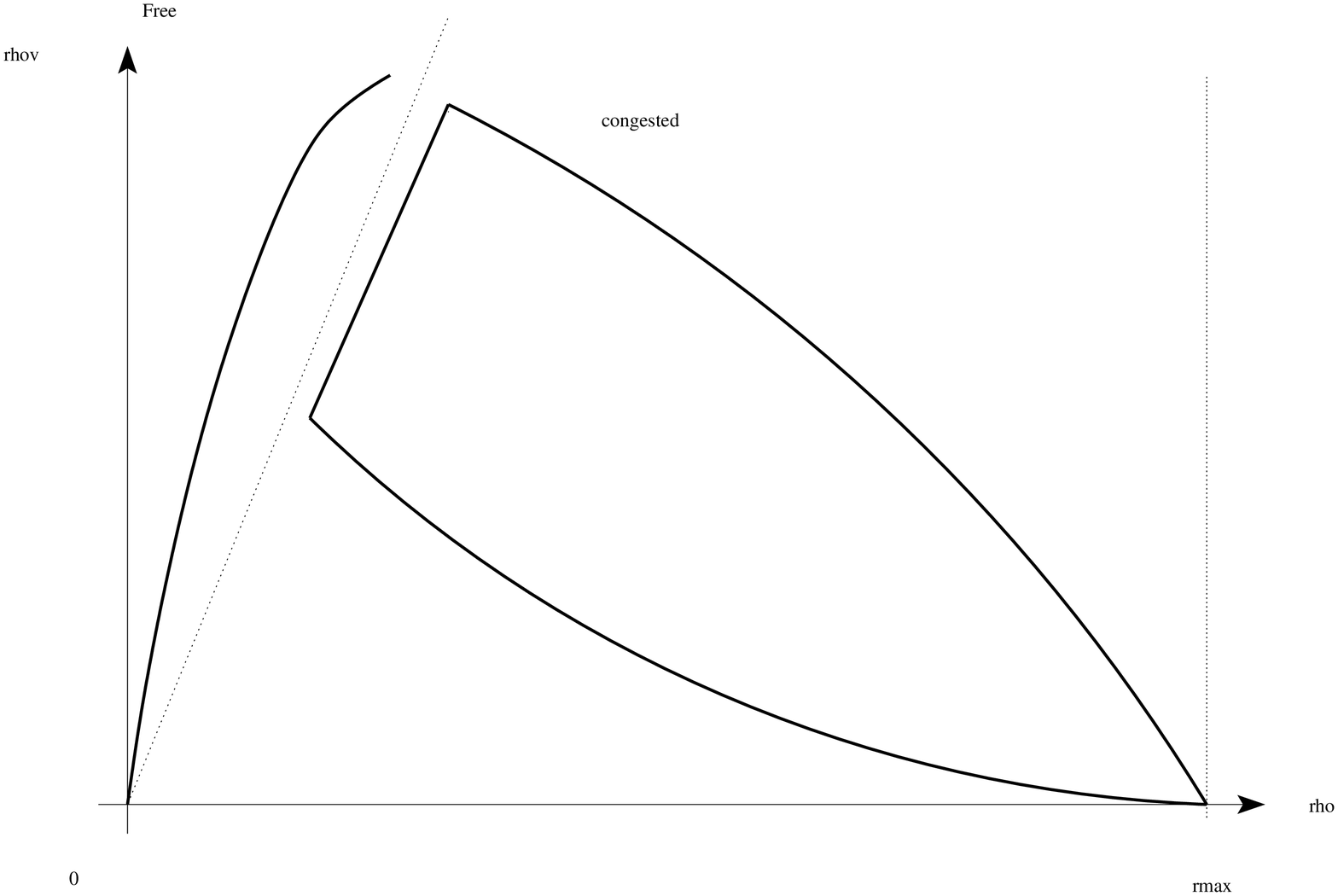}
  \end{psfrags}
  \caption{Fundamental diagrams, from left to right, of the (LWR)
    model~(\ref{eq:LWR}), of the (AR) model~(\ref{eq:AR}) and of the
    $2$-phase model~(\ref{eq:RC}).}
  \label{fig:fd}
\end{figure}
Indeed, compare Figure~\ref{fig:phases}, left with the measurements in
Figure~\ref{fig:Ex}.

As long as the data are in $F$, the solutions to~(\ref{eq:Modeleta})
are essentially the same as those of~(\ref{eq:LWR}). In the congested
phase, the solutions to~(\ref{eq:Modeleta}) obviously present a richer
structure, for they generically contain 2 waves instead of 1. In
particular, the (LWR) model~(\ref{eq:LWR}) may not describe the
\emph{"homogeneous-in-speed"} solutions, i.e.~a type of synchronized
flow, see~\cite[Section~2.2]{Kerner1} and~\cite{HelbingTreiber,
  WangJiangWuLiu}, which is described by the $2$-waves
in~(\ref{eq:Modeleta}).

Finally, note that if in~(\ref{eq:Modeleta}) the two geometric
parameters $\check w$ and $\hat w$ coincide, then we recover the
LWR~(\ref{eq:LWR}) model with $V(\rho) = \min \{ V_{\max}, \hat w\,
\psi(\rho) \}$.

\subsection{The Aw-Rascle Model}

Consider now the Aw--Rascle (AR) model
\begin{equation}
  \label{eq:AR}
  \left\{
    \begin{array}{l}
      \partial_t \rho + \partial_x \left[ \rho \, v(\rho,y) \right] =0
      \\
      \partial_t y + \partial_x \left[ y \, v(\rho,y) \right] =0
    \end{array}
  \right.
  \qquad
  v(\rho,y) = \frac{y}{\rho} - p(\rho)
\end{equation}
introduced in~\cite{AwRascle} and successively refined in several
papers, see for instan\-ce~\cite{AwKlarMaterneRascle,
  BagneriniRascle2003, Goatin2Phases, GodvikHanche, Greenberg,
  GreenbergKlarRascle, HertyKlar2003, MoutariRascle, SiebelMauser} and
the references therein.

Note that $w$ in~(\ref{eq:Modeleta}) plays a role analogous to that of
$v+p\left( \rho\right) $ in~(\ref{eq:AR}).

In the (AR) model, $R$ and the \emph{``pressure''} function need to be
selected, similarly to $R$ and $\psi$ in~(\ref{eq:Modeleta}). No other
parameter appears in~(\ref{eq:AR}), but the definition of an invariant
domain requires two parameters, with a role similar to that of $\check
w$ and $\hat w$. Indeed, an invariant domain for~(\ref{eq:AR}) is
\begin{displaymath}
  \left\{
    (\rho,y) \colon
    \rho \in [0,R] \mbox{ and }
    y \in \left[
      \rho\left(v_- + p(\rho) \right),
      \rho\left(v_- + p(\rho)
      \right)
    \right]
  \right\}
\end{displaymath}
see Figure~\ref{fig:fd}, center, and depends on the speeds $v_-$ and
$v_+$. More recent versions of~(\ref{eq:AR}) contain also a suitable
relaxation source term in the right hand side of the second equation;
in this case one more arbitrary function needs to be selected. The
original~(AR) model does not distinguish between a free and a
congested phase. However, it was extended to describe two different
phases in~\cite{Goatin2Phases}. Further comments on~(\ref{eq:AR}) are
found in~\cite{Lebacque}.

Concerning the analytical properties of the solutions, the Riemann
solver for the (AR) model suffers lack of continuous dependence at
vacuum, see~\cite[Section~4]{AwRascle}. However, existence of
solutions attaining also the vacuum state was proved
in~\cite{GodvikHanche}, while the $2$-phase construction
in~\cite{Goatin2Phases} also displays continuous dependence.

A qualitative difference between the (AR) model and the present one is
property~3.~in Proposition~\ref{prop:Q}. Indeed, solutions
to~(\ref{eq:AR}) may well have zero speed while being at a density
strictly lower than the maximal one.

\subsection{The Hyperbolic $2$-Phase Model}

Recall the model presented in~\cite{Colombo}, with a notation similar
to the present one:
\begin{equation}
  \label{eq:RC}
  \!\!\!
  \begin{array}{ll}
    \hbox{Free flow: } (\rho,q) \in F,
    &
    \hbox{Congested flow: } (\rho,q) \in C,
    \\
    \partial_t \rho +\partial_x \left[\rho \cdot v_F(\rho) \right] =0,
    \qquad
    &
    \left\{
      \begin{array}{l}
        \partial_t \rho + \partial_x \left[\rho\cdot v_C(\rho,q) \right] =0
        \\
        \partial_t q + \partial_x \left[ (q-q_*) \cdot v_C(\rho,q) \right] =0
      \end{array}
    \right.
    \\
    v_F(\rho) = \left( 1 - {\rho\over R} \right) \cdot V
    &
    v_C(\rho,q) = \left( 1 - {\rho\over R} \right) \cdot {q \over \rho}
  \end{array}
\end{equation}
the phases being defined as
\begin{displaymath}
  \begin{array}{lcl}
    F
    & = &
    \{ (\rho,q) \in [0,R] \times \reali^+ \colon
    v_f(\rho) \geq V_f ,\, q= \rho \cdot V \},
    \\[.5ex]
    C
    & = &
    \left\{ (\rho,q) \in [0,R] \times \reali^+ \colon
      v_c(\rho.q) \leq V_c ,\,
      \frac{q-q_*}{\rho} \in
      \left[
        \frac{Q_1-q_*}{R},
        \frac{Q_2-q_*}{R}
      \right]
    \right\}.
  \end{array}\vspace{2mm}
\end{displaymath}
In~(\ref{eq:RC}) no function can be selected, on the other hand the
evolution depends on the parameters $V$, $R$ and $q_*$ while the
invariant domains $F$ and $C$ depend on $V_f$, $V_c$, $Q_1$ and
$Q_2$. A geometric construction of the solutions to~(\ref{eq:RC}) in
the congested phase is in~\cite{LebacqueEtAl}.

The main difference between fundamental diagrams of~(\ref{eq:RC}), see
Figure~\ref{fig:fd}, right, and that of~(\ref{eq:Modeleta}) is
that~(\ref{eq:RC}) requires the two phases to be \emph{disconnected}:
there is a \emph{gap} between the free and the congested phase. This
restriction is necessary for the well posedness of the Riemann problem
for~(\ref{eq:RC}) and can be hardly justified on the basis of
experimental data. More recently, the global well-posedness of the
model~(\ref{eq:RC}) was proved in~\cite{ColomboGoatinPriuli}.

Note that in both models, as well as in that presented
in~\cite{Goatin2Phases}, the free phase is one dimensional, while the
congested phase is bidimensional.

The model~(\ref{eq:RC}) allows for the description of \emph{wide
  jams}, i.e.~of persistent waves in the congested phase moving at a
speed \emph{different} from that of traffic. Here, as long as
$\frac{d^2\ }{d\rho^2} \left( \rho \, \psi(\rho) \right) < 0$,
persistent phenomena can be described only through waves of the second
family, which move at the mean traffic speed. We refer
to~\cite{Lebacque} for further discussions on~(\ref{eq:RC}) and
comparisons with other macroscopic models.

\subsection{A Kinetic Model}
\label{sec:Kin}

Recall, with a notation adapted to the present case, the kinetic model
introduced in~\cite[Section~1]{BenzoniColomboGwiazda}:
\begin{equation}
  \label{eq:Kin}
  \partial_t r(t,x;w)
  +
  \partial_x \left[
    w \, r(t,x;w) \,
    \psi\left( \int_{\check w}^{\hat w} r(t,x;w') \, d w' \right)
  \right]
  =
  0
  \,.
\end{equation}
The function $\psi$ and the speed $w$ play the same role as here. The
unknown $r = r(t,x;w)$ is the probability density of vehicles having
maximal speed $w$ that at time $t$ are at point $x$.

In~(\ref{eq:Kin}) there is one function to be specified, $\psi$, as
in~(\ref{eq:Modeleta}). The parameters are $R$ (which is normalized to
$1$ in~\cite{BenzoniColomboGwiazda}), $\check w$ and $\hat w$,
similarly to~(\ref{eq:Modeleta}). Since no limit speed is there
defined, no parameter in~(\ref{eq:Kin}) has the same role as here
$V_{\max}$.

Being of a kinetic nature, there is no real equivalent to a
fundamental diagram for~(\ref{eq:Kin}).

\par From the analytical point of view, the existence of solutions
to~(\ref{eq:Kin}) has not been proved, yet.  The main result
in~\cite{BenzoniColomboGwiazda} only states that~(\ref{eq:Kin}) can be
rigorously obtained as the limit of systems of $n \times n$
conservation laws describing $n$ populations of vehicles, each
characterized by their maximal speed.

Let the measure $r$ solve~(\ref{eq:Kin}) and be such that for suitable
functions $\rho$ and $w$
\begin{equation}
  \label{eq:delta}
  r(t,x; \cdot ) = \rho(t,x) \, \delta_{w(t,x)}
\end{equation}
where $\delta$ is the usual Dirac measure. Then, formally, $(\rho,w)$
solves~(\ref{eq:Modeleta}). Indeed, for the first equation simply
substitute~(\ref{eq:delta}) in~(\ref{eq:Kin}) and integrate; for the
second equation substitute~(\ref{eq:delta}) in~(\ref{eq:Kin}),
multiply by $w$ and integrate over $[\check w, \hat w]$.

Remark that~(\ref{eq:delta}) suggests a further interpretation of the
quantity $\rho$ in~(\ref{eq:Modeleta}). Indeed, in the present model,
at $(t,x)$ vehicles of only one species are present, namely those with
maximal speed $w(t,x)$.

\Section{Connections with a Follow-The-Leader Model}
\label{sec:Micro}

Within the framework of~(\ref{eq:rhow}), a single driver starting from
$\tilde p$ at time $t=0$ follows the \emph{particle path} $p = p(t)$
that solves the Cauchy problem
\begin{equation}
  \label{eq:PP}
  \left\{
    \!\!
    \begin{array}{l}
      \dot p
      =
      v \left( \rho\left(t,p(t) \right), w\left((t,p(t)\right) \right)
      \\
      p(0) = \tilde p
    \end{array}
  \right.
  \quad
  v(\rho,w) = \min \left\{V_{\max} ,\, w \, \psi(\rho) \right\},
\end{equation}
refer to~\cite{ColomboMarson} for the well posedness of the particle
path for the LWR model (see also~\cite{AwKlarMaterneRascle}).  Recall
now that $w$ is a specific feature of every single driver,
i.e.~$w\left( t, p(t) \right) = w(0,\tilde p)$ for all $\tilde p$. On
the other hand, from a microscopic point of view, if $n$ drivers are
distributed along the road, then $\rho$ is approximated by $l /
(p_{i+1} - p_i)$, where $l$ is a standard length of a car.

We fix $L > 0$ and assume that $n$ drivers are distributed along
$[-L,L]$. Then, the natural microscopic counterpart
to~(\ref{eq:Modeleta}) is therefore the \textit{Follow-The-Leader}
(FTL) model defined by the Cauchy problem
\begin{equation}
  \label{eq:ftl}
  \left\{
    \begin{array}{l@{\qquad}rcl}
      \dot p_i
      =
      v \left( \frac{l}{p_{i+1} - p_i} , w_i\right)
      & i & = & 1, \ldots, n
      \\
      \dot p_{n+1}
      =
      V_{\max}
      \\[5pt]
      p_{i}(0)
      =
      \tilde p_{i}
      & i & = & 1, \ldots, n+1
    \end{array}
  \right.
\end{equation}
where $\tilde p_1 =-L$ and $\tilde
p_{n+1}=L-l$. Proposition~\ref{prop:Cauchy problem} shows
that~(\ref{eq:ftl}) admits a unique global solution defined for every
$t \geq 0$ and such that $p_{i+1} - p_i \geq l$ for all $t \geq 0$.

\begin{proposition}
  \label{prop:Cauchy problem}
  Let~\textbf{a.}, \textbf{b.} and~\textbf{c.} hold. Fix $L>0$. For
  any $n \in \naturali$, with $n\geq 2$, choose initial data $\tilde
  p^n_i$ for $i=1, \ldots, n$ satisfying $\tilde p^n_{i+1} - \tilde
  p^n_i \geq l$. Then, the Cauchy problem~(\ref{eq:ftl}) admits a
  unique solution $p^n_i = p^n_i(t)$, for $i=1, \ldots, n+1$, defined
  for all $t \geq 0$ and satisfying $p^n_{i+1}(t) - p^n_i(t) \geq l$
  for all $t \geq 0$ and for $i=1, \ldots, n$ .
\end{proposition}

\noindent The proof is postponed to Section~\ref{sec:Tech}.

Our next aim is to rigorously show that in the limit $n \to +\infty$
with $n \, l = \mbox{constant}>0$, the microscopic model
in~(\ref{eq:ftl}) yields the macroscopic one
in~(\ref{eq:Modeleta}). Given the position $p^i$ of every single
vehicle and its maximal speed $w_i$, for $i=1, \ldots, n+1$, the
macroscopic variables $\rho, w$ are given by
\begin{displaymath}
  \rho(x)
  =
  \sum_{i=1}^n \frac{l}{p^n_{i+1} - p^n_i} \;
  \caratt{[p^n_i, p^n_{i+1}[} (x)
  \quad \mbox{ and } \quad
  w(x)
  =
  \sum_{i=1}^n
  w^n_i \; \caratt{[p^n_i, p^n_{i+1}[} (x)
  \,.
\end{displaymath}
Note that necessarily $p^n_{i+1} - p^n_i \geq l$.

On the contrary, given $( \rho, w) \in (\L1 \cap \BV) (\reali; [0,1]
\times [\check w, \hat w])$, with $\supp \rho$, $\supp w \subseteq
[-L,L]$, we reconstruct a microscopic description defining $l = \left(
  \int_{\reali} \rho(x) \, dx \right) /n$ and
\begin{eqnarray*}
  p^n_{n+1} & = & L-l
  \\
  p^n_{i}
  & = &
  \max \left\{
    p \in [-L,L] \colon \int_{p}^{p_{i+1}} \rho(x) \, dx = l
  \right\}
  \quad \mbox{ for } i = 1, \ldots, n
  \\
  w^n_i
  & = &
  w(p^n_i+)
  \quad \mbox{ for } i = 1, \ldots, n+1 \,.
\end{eqnarray*}
\noindent Note that $\int_{\reali}\rho(x) \, dx = nl>0$. Now we are
able to rigorously show that, as the number of vehicles increases to
infinity, the microscopic model in~(\ref{eq:ftl}) yields the
macroscopic one in~(\ref{eq:Modeleta}).

\begin{proposition}
  \label{prop:limit}
  Let~\textbf{a.}, \textbf{b.} and~\textbf{c.} hold. Fix $T >
  0$. Choose $(\tilde \rho, \tilde w) \in (\L1 \cap \BV) (\reali;
  [0,1] \times[\check w, \hat w])$ with $\supp \tilde\rho$, $\supp
  \tilde w \subseteq [-L,L]$. Construct the initial data for the
  microscopic model setting $l = \left( \int_{\reali} \tilde \rho(x)
    \, dx \right) /n$ and
  \begin{eqnarray*}
    \tilde p^n_{n+1} & = & L-l
    \\
    \tilde p^n_{i}
    & = &
    \max \left\{
      p \in [-L,L] \colon \int_{p}^{\tilde p_{i+1}} \tilde \rho(x) \, dx = l
    \right\}
    \quad \mbox{ for } i = 1, \ldots, n
    \\
    \tilde w^n_i
    & = &
    \tilde w(p^n_i+)
    \quad \mbox{ for } i = 1, \ldots, n+1 \,.  
  \end{eqnarray*}
  Let $p^n_i(t)$, for $i=1, \ldots, n$, be the corresponding solution
  to~(\ref{eq:ftl}). Define
  \begin{eqnarray}
    \label{eq:rhon}
    \rho^n(t,x)
    & = &
    \sum_{i=1}^n \frac{l}{p^n_{i+1}(t) - p^n_i(t)} \;
    \caratt{[p^n_i(t), p^n_{i+1}(t)[} (x)
    \\
    \label{eq:wn}
    w^n(t,x)
    & = &
    \sum_{i=1}^n
    \tilde w^n_i \; \caratt{[p^n_i(t), p^n_{i+1}(t)[} (x)
    \,.
  \end{eqnarray}
  If there exists a pair $(\rho,w) \in \L\infty \bigl( [0,T]; \L1
  (\reali; [0,1] \times [\check w, \hat w] \bigr)$ such that
  \begin{displaymath}
    \lim_{n \to +\infty} (\rho^n,w^n)(t,x) = (\rho,w)(t,x)
    \qquad p.a.e.
  \end{displaymath}
  then, the pair $(\rho,\rho w)$ is a weak solution
  to~(\ref{eq:Modeleta}) with initial datum $(\tilde \rho, \tilde
  \rho\tilde w)$.
\end{proposition}
The proof is postponed to Section~\ref{sec:Tech}.

\Section{Technical Details}
\label{sec:Tech}

We first prove an elementary consequence of our
assumption~\textbf{\ref{it:h2}}.

\begin{lemma}
  \label{lem:psi}
  Let $\psi$ satisfy~\textbf{\ref{it:h2}}. Then,
  \begin{displaymath}
    \exists\,\bar\rho \in \left[0,R\right[
    \mbox{such that }
    \left\{
      \begin{array}{l}
        \psi \mbox{ is constant on } [0, \bar\rho],
        \\
        \psi \mbox{ is strictly decreasing on } [\bar\rho, R].
      \end{array}
    \right.
  \end{displaymath}
\end{lemma}

\begin{proof}
  Call $q(\rho) = \rho\, \psi(\rho)$.  If $\psi$ is strictly monotone,
  then $\bar\rho =0$ and the proof is completed. Otherwise, assume
  that $\psi(\rho_1) = \psi(\rho_2) = c$ for suitable $\rho_1,\rho_2
  \in \left]0, R\right]$ and $\rho_1 \neq \rho_2$. Then,
  by~\textbf{\ref{it:h2}.}, for all $\rho \in [\rho_1,\rho_2]$ we have
  $\psi(\rho) = c$ and $q(\rho) = c\rho$. If $\psi(0) = c$, then the
  proof is completed. Otherwise, note that $q'(0) = \psi(0) >c$
  contradicts the convexity of $q$.
\end{proof}

\begin{corollary}
  \label{lem:Unique}
  Let $\psi$ satisfy~\textbf{\ref{it:h2}.}
  and~\textbf{\ref{it:h3}}. Then,
  \begin{displaymath}
    \bar\rho
    <
    \min
    \left\{
      \rho \in [0, R] \colon \exists w \in [\check w, \hat
      w] \mbox{ such that } (\rho,w) \in C
    \right\} \,.
  \end{displaymath}
\end{corollary}

\noindent The proof is immediate and, hence, omitted.

In the sequel, for the basic definitions concerning the standard
theory of conservation laws we refer to~\cite{BressanLectureNotes}.

\begin{proofof}{Theorem~\ref{thm:RP}}
  We consider different cases, depending on the phase of the
  data~(\ref{eq:RD}).

  \smallskip
  \noindent\textbf{1.} \quad $(\rho^l,\eta^l), (\rho^r, \eta^r) \in
  F$.

  In this case, (\ref{eq:Modeleta}) reduces to the degenerate linear
  system~(\ref{eq:F}) so that the
  problem~(\ref{eq:Modeleta})--(\ref{eq:RD}) is solved
  by~(\ref{eq:solF}). Remark, for later use, that the characteristic
  speed is $\lambda^F = V_{\max}$.

  \smallskip
  \noindent\textbf{2.} \quad $(\rho^l,\eta^l), (\rho^r, \eta^r) \in
  C$.

  Now, $v(\rho, \eta) = \eta \, \psi(\rho) / \rho$. We show that $C$
  is invariant with respect to the $2\times 2$ system of conservation
  laws~(\ref{eq:C}).  To this aim, we compute the eigenvalues, right
  eigenvectors and the Lax curves in $C$:
  \begin{displaymath}
    \begin{array}{rcl@{\quad}rcl}
      \lambda_{1} (\rho, \eta)
      & = &
      \eta\, \psi'(\rho) + v(\rho, \eta)
      &
      \lambda_{2} (\rho, \eta)
      & = &
      v(\rho, \eta)
      \\[5pt]
      r_{1} (\rho, \eta)
      & = &
      \left[
        \begin{array}{c}
          -\rho
          \\
          -\eta
        \end{array}
      \right]
      &
      r_{2} (\rho, \eta)
      & = &
      \left[
        \begin{array}{c}
          1
          \\
          \eta\left( \frac{1}{\rho}-\frac{\psi'(\rho) }{\psi(\rho) }\right)
        \end{array}%
      \right]
      \\
      \nabla \lambda_1 \cdot r_1
      & = &
      \displaystyle
      -\frac{d^2\ }{d\rho^2} \left[ \rho\, \psi(\rho) \right]
      &
      \nabla \lambda_2 \cdot r_2
      & = &
      0
      \\
      \mathcal{L}_1(\rho;\rho_o,\eta_o)
      & = &
      \displaystyle
      \eta_o \frac{\rho}{\rho_o}
      &
      \mathcal{L}_2(\rho;\rho_o,\eta_o)
      & = &
      \displaystyle
      \frac{\rho \, v(\rho_o, \eta_o)}{\psi(\rho)},
      \; \rho_o < R.\!
    \end{array}
  \end{displaymath}
  When $\rho_o = R$, the $2$--Lax curve through $(\rho_o, \eta_o)$ is
  the segment $\rho=R$, $\eta \in [R \check w, R \hat w]$.

  Shock and rarefaction curves of the first characteristic family
  coincide by~\cite[Lemma~2.1]{BagneriniColomboCorli}, see
  also~\cite[Problem~1, Chapter~5]{BressanLectureNotes}. The second
  characteristic field is linearly degenerate. Hence, (\ref{eq:C}) is
  a Temple system and $C$ is invariant, since its boundary consists of
  Lax curves, see~\cite[Theorem~3.2]{Hoff}.

  Thus, the solution to~(\ref{eq:Modeleta}) is as described
  in~\textsl{(2)} and attains values in $C$.

  \smallskip
  \noindent\textbf{3.} \quad $(\rho^l,\eta^l) \in C$, $(\rho^r,
  \eta^r) \in F$.

  Let $\rho^m$ satisfy $\psi(\rho^m) = V_{\max} \rho^r / \eta^r$. Note
  that such $\rho^m$ exists in $\left]0,1\right[$
  by~\textbf{\ref{it:h2}} and~\textbf{\ref{it:h3}.}, it is unique by
  Corollary~\ref{lem:Unique}.  Define $\eta^m = (\rho^m/\rho^r)
  \eta^r$ and note that $(\rho^l,\eta^l)$, $(\rho^m,\eta^m)$ are
  connected by a $1$--rarefaction wave of~(\ref{eq:C}) having maximal
  speed of propagation $\lambda_1(\rho^m, \eta^m) < V_{\max}$. Hence,
  a linear wave, solution to~(\ref{eq:F}), can be juxtaposed
  connecting $(\rho^m, \eta^m)$ to $(\rho^l, \eta^l)$ and the solution
  to~(\ref{eq:Modeleta}) is as described at~\textsl{(3)}.

  \smallskip
  \noindent\textbf{4.} \quad $(\rho^l,\eta^l) \in F$, $(\rho^r,
  \eta^r) \in C$ (see Figure~\ref{fig:phase}).

  Note that system~(\ref{eq:C}) can be considered on the whole of $F
  \cup C$. Also this set is invariant for~(\ref{eq:C}),
  by~\cite[Theorem~3.2]{Hoff}. Then, in this case, we let
  $(\rho,\eta)$ be the standard Lax solution to~(\ref{eq:C}), as
  described at~\textsl{(4)}.
\end{proofof}

\begin{proofof}{Proposition~\ref{prop:RS}}
  We consider different cases depending on the phase of the
  data~(\ref{eq:RD}).
 
  If $(\rho^l,\eta^l), (\rho^r,\eta^r) \in F$, then $\mathcal{R}
  \left((\rho^l,\eta^l), (\rho^r,\eta^r) \right)$ coincides with the
  Riemann solver of a linear system, which
  satisfies~\textbf{(C1)}. Condition~\textbf{(C2)} is immediate since
  no nontrivial middle state is available.

  Similarly, if $(\rho^l,\eta^l), (\rho^r,\eta^r) \in C$, then
  $\mathcal{R} \left((\rho^l,\eta^l), (\rho^r,\eta^r) \right)$
  coincides with the standard Riemann solver of a $2\times 2$ system,
  which is consistent. The consistency of $\mathcal{R}$ then follows
  by the invariance of $C$, by~\textbf{2.} in the proof of
  Theorem~\ref{thm:RP}.

  By the same argument, also the case $(\rho^l,\eta^l)\in F$ and
  $(\rho^r,\eta^r) \in C$ is proved. Indeed, in ~\textbf{(C2)}, note
  that the only possible nontrivial middles states are in $C$.

  Finally, if $(\rho^l,\eta^l) \in C$ and $(\rho^r,\eta^r) \in F$,
  then $\mathcal{R} \left((\rho^l,\eta^l), (\rho^r,\eta^r) \right)$
  takes values in $F \cup C$ and is the juxtaposition of 2 consistent
  Riemann problems, hence~\textbf{(C1)}
  holds. Concerning~\textbf{(C2)}, note that the the only possible
  nontrivial middles states are in $C$, and~\textbf{(C2)} follows by
  the consistency of the standard Riemann solver for~(\ref{eq:C}).

  Thus 1. is proved. Assertions~ 2. and 3. are immediate consequences
  of the construction of Theorem~\ref{thm:RP}.

  Assume now that $\mathcal{R}$ satisfies~2 and~3. Then all Riemann
  problems with data $(\rho^l,\eta^l), (\rho^r,\eta^r) \in F$,
  $(\rho^l,\eta^l) \in F$, $(\rho^r,\eta^r) \in C$ and
  $(\rho^l,\eta^l), (\rho^r,\eta^r) \in C$ are uniquely solved. The
  solution to Riemann problems with $(\rho^l,\eta^l) \in C$ and
  $(\rho^r,\eta^r) \in F$ is then uniquely constructed
  through~\textbf{(C1)}.
\end{proofof}

\begin{proofof}{Proposition~\ref{prop:Q}}
  Consider the different statements separately.

  1.\quad The invariance of $F$, $C$ and $F \cup C$ is shown in the
  proof of Theorem~\ref{thm:RP}.

  2.\quad By the invariance of $F \cup C$, it is sufficient to observe
  that on the compact set $F \cup C$, the density $\rho$, respectively
  the speed $v$, is uniformly bounded by $R$, respectively $V_{\max}$.

  3.\quad It is immediate, see for instance Figure~\ref{fig:phases},
  left.

  4.\quad In phase $C$ we have
  \begin{displaymath}
    \lambda_{1}(\rho,\eta)
    =
    \eta \, \psi'(\rho) + v(\rho, \eta)
    \leq
    v(\rho, \eta)
    \quad \mbox{ and } \quad
    \lambda_{2} (\rho,\eta)
    \leq
    v(\rho, \eta).
  \end{displaymath}
  In the free phase the wave speed is $V_{\max} = v(\rho,\eta)$. The
  only case left is that of a phase boundary with left state in $F$
  and right state, say $(\rho^r,\eta^r)$, in $C$. Then, the speed
  $\Lambda$ of the phase boundary clearly satisfies $\Lambda \leq
  \lambda_1(\rho^r,\eta^r) < v(\rho^r,\eta^r)$.
\end{proofof}

\begin{proofof}{Proposition~\ref{prop:Cauchy problem}}
  Note first that the functions $\rho \to v( \rho, w_{i}) $
  in~(\ref{eq:ftl}) are uniformly bounded and Lipschitz continuous for
  $i = 1, \ldots, n$. We extend them to functions with the same
  properties and defined on $\left[0, +\infty\right[$ setting
  \begin{equation}
    \label{eq:ftl3}
    u_{i} (\rho)
    =
    \left\{
      \begin{array}{ll}
        V_{\max} \quad \quad if \qquad  \rho <0
        \\
        v\left(\rho,w_{i}\right) \quad if \qquad \rho \in\left[ 0,1\right]
        \\
        0 \qquad \qquad  if \qquad \rho >1.
      \end{array}
    \right.
  \end{equation}
  We also note that, for $i = 1, \ldots, n$, the composite
  applications $\delta \to u_{i} ( l / \delta )$, can be extended to
  uniformly bounded and Lipschitz continuous functions on $\left[
    0,+\infty\right[$. Now we consider the Cauchy problem
  \begin{equation}
    \label{eq:CP}
    \left\{
      \begin{array}{l@{\qquad}rcl}
        \dot p^n_i
        =
        u_{i}\left( \frac{l}{p^n_{i+1} - p^n_i} \right)
        & i & = & 1, \ldots, n
        \\
        \dot p^n_{n+1}
        =
        V_{\max}
        \\[5pt]
        p^n_{i}(0)
        =
        \tilde p_{i}
        & i & = & 1, \ldots, n+1 \,.
      \end{array}
    \right.
  \end{equation}
  Note that $\tilde p^n_i$, for $i=1, \ldots, n+1$ are defined in
  Proposition~\ref{prop:limit} and satisfy the condition $\tilde
  p^n_{i+1} - \tilde p^n_i \geq l > 0$, for every $i=1,\ldots,n$.

  By the standard ODE theory, there exists a $\C1$ solution
  $p_{i}^{n}$ defined as long as $p^n_{i+1} - p^n_i > 0$ for all
  $i = 1, \ldots, n$.  We now prove that in fact $p^n_{i+1}(t) -
  p^n_i(t) \geq l$ for every $t \geq 0$. To this aim we assume by
  contradiction that there exist positive $\bar t$ and $t^{*}$, with
  $\bar t<t^{*}$, such that $p^n_{i+1}(\bar t) - p^n_i(\bar t) = l$
  and $p^n_{i+1}(t) -p^n_i(t) < l$ for every $t \in \left] \bar t,
    t^{*} \right]$. Then,
  \begin{displaymath}
    p^n_i(t)
    =
    p^n_i(\bar t) + \int_{\bar t}^{t} \! \dot p_i (s) \, ds
    =
    p^n_i(\bar t) + \int_{\bar t}^{t}
    u_{i} \left( \frac{l}{p^n_{i+1}(s) - p^n_i(s)} \right) ds
    =
    p^n_i(\bar t) .
  \end{displaymath}
  This yields a contradiction, since for every $t \in \left] \bar t,
    t^{*} \right]$
  \begin{eqnarray*}
    p^n_{i+1}(t) -p^n_i(t)
    \geq
    p^n_{i+1}(\bar t) - p^n_i( \bar t)
    =
    l\,,
  \end{eqnarray*}
  completing the proof.
\end{proofof}

\begin{proofof}{Proposition~\ref{prop:limit}}
  Recall first the definition of weak solution to~(\ref{eq:Modeleta}):
  for all $\phi \in \Cc\infty$, setting $v(\rho,w) = \min \{V_{\max},
  \, w \, \psi(\rho) \}$,
  \begin{displaymath}
    \int_0^T \int_{\reali}
    \left(
      \left[\!\!
        \begin{array}{c}
          \rho \\ \rho \, w
        \end{array}
        \!\!\right]
      \partial_t \phi
      +
      \left[\!\!
        \begin{array}{c}
          \rho \, v(\rho,w)
          \\
          \rho \, w \, v(\rho,w)
        \end{array}
        \!\!\right]
      \partial_x \phi
    \right)
    \! dx \, dt
    +
    \int_{\reali}
    \left[\!\!
      \begin{array}{c}
        \tilde \rho \\ \tilde \rho \, \tilde w
      \end{array}
      \!\!\right]
    \phi(0,x) \, dx
    =
    0
  \end{displaymath}
  and consider the two components separately.

  Below, $\O$ denotes a constant that uniformly bounds from above the
  modulus of $\phi$ and all its derivatives up to the second order.
  Insert first~(\ref{eq:rhon}) in the above equality and obtain:
  \begin{eqnarray*} 
    I^n\!\!
    & := &
    \int_0^T \int_{\reali}
    \left(
      \rho^n
      \partial_t \phi
      +
      \rho^n \, v(\rho^n,w^n) \,
      \partial_x \phi
    \right)
    \, dx \, dt
    +
    \int_{\reali}
    \tilde \rho \, \phi(0,x) \, dx
    \\
    & = &
    \! \sum_{i=1}^n
    \!\int_0^T \!\!
    \frac{l}{p^n_{i+1}(t) - p^n_i(t)}
    \! \int_{p^n_i(t)}^{p^n_{i+1}(t)}
    \!\! \left[
      \partial_t \phi
      +
      v \! \left(\frac{l}{p^n_{i+1}(t) - p^n_i(t)}, w^n_i \right)
      \! \partial_x \phi
    \right]
    \! dt
    \\
    & &
    \qquad
    +
    \int_{\reali}
    \rho^n(0,x) \, \phi(0,x) \, dx
    +
    \int_{\reali}
    \left( \tilde \rho -\rho^n(0,x) \right)\, \phi(0,x) \, dx
    \\
    & = &
    \sum_{i=1}^n
    \int_0^T
    \frac{l}{p^n_{i+1}(t) - p^n_i(t)}
    \int_{p^n_i(t)}^{p^n_{i+1}(t)} \!\!
    \left(
      \partial_t \phi(t,x) + \dot p^n_i(t) \partial_x \phi(t,x)
    \right)
    dx \, dt
    \\
    & & \qquad
    +
    \sum_{i=1}^n
    \frac{l}{\tilde p^n_{i+1} - \tilde p^n_i}
    \int_{\tilde p_i}^{\tilde p_{i+1}} \! \! \phi(0,x) dx
    +
    \int_{\reali}
    \left( \tilde \rho -\rho^n(0,x) \right) \phi(0,x) \, dx.
  \end{eqnarray*}
  Approximating $\phi\left( t,x\right) $ with $\phi\left(
    t,p^n_i(t)\right) $ for every $x$ in $[p^n_i(t), p^n_{i+1}(t)]$,
  we obtain:
  \begin{eqnarray*} 
    I^n
    & = &
    \sum_{i=1}^n
    \int_0^T
    \frac{l}{p^n_{i+1}(t) - p^n_i(t)}
    \int_{p^n_i(t)}^{p^n_{i+1}(t)}
    \frac{d}{dt}\phi\left(t,p^n_i(t) \right) \, dx \, dt
    \\
    & & \quad
    +
    \sum_{i=1}^n
    \int_0^T
    \frac{l}{p^n_{i+1}(t) - p^n_i(t)}
    \int_{p^n_i(t)}^{p^n_{i+1}(t)}
    \O \left( p^n_{i+1}(t) - p^n_i(t) \right) \, dx \, dt
    \\
    & & \quad
    +
    \sum_{i=1}^n
    \frac{l}{\tilde p^n_{i+1} - \tilde p^n_i}
    \int_{\tilde p_i}^{\tilde p_{i+1}} \! \phi(0,x) \, dx
    +
    \int_{\reali} \!
    \left( \tilde \rho -\rho^n(0,x) \right) \phi(0,x) \, dx
    \\
    & = &
    l
    \sum_{i=1}^n
    \int_0^T
    \frac{d}{dt}\phi\left(t,p^n_i(t) \right) dt
    +
    \Delta x
    \sum_{i=1}^n
    \int_0^T \!
    \O \left( p^n_{i+1}(t) - p^n_i(t) \right) dx \, dt
    \\
    & & \quad
    +
    \sum_{i=1}^n
    \frac{l}{\tilde p^n_{i+1} - \tilde p^n_i}
    \int_{\tilde p_i}^{\tilde p_{i+1}} \phi(0,x) dx
    +
    \int_{\reali} \!
    \left( \tilde \rho -\rho^n(0,x) \right) \phi(0,x) \, dx
    \\
    & = &
    \sum_{i=1}^n  
    \frac{l}{\tilde p^n_{i+1} - \tilde p^n_i}
    \int_{\tilde p_i}^{\tilde p_{i+1}}
    \left[\phi(0,x) - \phi(0,\tilde p^n_i) \right] \, dx
    \\
    & & \quad
    +
    \O \, l \left( p^n_{n+1}(T) - p^n_1(T) \right)
    +
    \int_{\reali}
    \left( \tilde \rho -\rho^n(0,x) \right)\, \phi(0,x) \, dx
    \\
    & = &
    \O \, l \left(2L+V_{\max}T\right)
    +
    \int_{\reali}
    \left( \tilde \rho -\rho^n(0,x) \right)\, \phi(0,x) \, dx
  \end{eqnarray*}
  and both terms in the latter quantity clearly vanish as $n \to
  +\infty$.

  The computations related to the other component are entirely
  similar, since $w$ is constant along any set of the form
  \begin{displaymath}
    \left\{
      (t,x) \in [0,T] \times \reali \colon
      x \in \left[p^n_i(t), p^n_{i+1}(t)\right[
    \right\}
  \end{displaymath}
  and the proof is completed.
\end{proofof}

\begin{remark}
  \label{rem:appendix}
  System~(\ref{eq:NonCons}) is not in conservation form. As far as
  smooth solutions are concerned, it is equivalent to infinitely many
  $2\times 2$ systems of conservation laws. Indeed, introduce a
  strictly monotone function $f \in \C2 \left([\check w, \hat w];
    \left]0, +\infty\right[ \right)$. Then, elementary computations
  show that, as long as smooth solutions are concerned,
  system~(\ref{eq:NonCons}) is equivalent to
  \begin{equation}
    \label{eq:General}
    \left\{
      \begin{array}{l}
        \partial_t \rho
        +
        \partial_x  \left( \rho \, \psi(\rho) \, g(\eta/\rho)\right)
        =0
        \\
        \partial_t \eta
        +
        \partial_x  \left( \eta \, \psi(\rho) \, g(\eta/\rho)\right)
        =0
      \end{array}
    \right.
    \qquad
    \mbox{ where }
    \qquad
    \begin{array}{l}
      \eta = \rho \, f(w) \mbox{ and}
      \\
      g\left(f(w)\right)=w  
    \end{array}
  \end{equation}
  Clearly, different choices of $f$ yield different weak solutions
  to~(\ref{eq:General}), but they are all equivalent when written in
  terms of $\rho$ and $w$.
\end{remark}

\noindent\textbf{Acknowledgment.} The first and second author
acknowledge the warm hospitality of the Laboratoire de Mathematiques
J.A.~Dieudonne of the Universit\`e Sophia-Antipolis de Nice, where
part of this work was completed.

{\small{

    \bibliographystyle{abbrv}

    \bibliography{model} }}

\end{document}